\documentclass{article}
\usepackage{amsmath}
\usepackage{amssymb}
\usepackage{amsthm}
\newtheorem{lem}{Lemma}
\newtheorem{thm}{Theorem}
\newtheorem{prop}{Proposition}
\newtheorem{defn}{Definition}
\newtheorem*{rem}{Remark}
\begin{document}

\title{Completely Order Bounded Maps on Non-Commutative \(L_p\)-Spaces}
\author{Erwin Neuhardt}

\date{December 5, 2019}
\maketitle
Mathematics  Subject  Classification  (2010): Primary 46L07; Secondary 47L25

Keywords: non-commutative \(L_p\)-space, matrix norm, completely positive\\
\noindent\hspace*{5mm}%
map, decomposable map

\begin{abstract}
We define norms on \(L_p(\mathcal{M}) \otimes M_n\)  where \(\mathcal{M}\) is a von Neumann algebra and \(M_n\) is the complex \(n \times n\) matrices.  We show that a linear map \(T: L_p(\mathcal{M}) \to L_q(\mathcal{N})\) is decomposable if \(\mathcal{N}\) is an injective von Neumann algebra, the maps \(T \otimes Id_{M_n}\) have a common upper bound with respect to our defined norms, and \(p = \infty\) or \(q = 1\). For \(2p < q < \infty\) we give an example of a map \(T\) with uniformly bounded maps \(T \otimes Id_{M_n}\) which is not decomposable.
\end{abstract}
\section{INTRODUCTION}
Completely positive maps on von Neumann algebras have been studied extensively and there are some nice results on such maps (see e.g. \cite{Ta}, Ch. IV.3 and \cite{Pi-OpSp}, Ch. 11). When combining the order structure with the vector space structure, it is natural to investigate the linear span of completely positive maps. These are called decomposable maps in \cite{Ha1}, \S1. The decomposable maps from a von Neumann algebra  \(\mathcal{M}\) into an injective von Neumann algebra \(\mathcal{N}\) are the completely bounded maps. This has been shown at about the same time by Haagerup \cite{Ha1}, Paulsen \cite{Pa}, and Wittstock \cite{Wi}. Completely bounded maps use the operator norm on matrices of elements of a von Neumann algebra.

For a von Neumann algebra \(\mathcal{M}\), the non-commutative \(L_p\)-space \(L_p(\mathcal{M}),\) \(1 \le p < \infty,\) can be realized  as (in general) unbounded operators on an Hilbert space. Therefore the \(L_p\)-space itself as well as matrices with elements of an \(L_p\)-space have a natural order given by positive operators. This allows us to define completely positive maps and decomposable maps from one \(L_p\)-space into another one. Then there should be a description of decomposable maps using some norm conditions on matrices of \(L_p\)-spaces.

This question has been partially answered by Pisier \cite{Pi} for linear maps from the Schatten classes \(S_p\) to \(S_p\),   and Arhancet and Kriegler \cite{ArKr}  for  linear maps from \(L_p(\mathcal{M})\) to \(L_p(\mathcal{N})\), where \( \mathcal{M}\) and \(\mathcal{N}\) are semifinite, approximately finite dimensional von Neumann algebras. Junge and Ruan \cite{JuRu} give a description of the decomposable norm for finite rank maps from \(L_p(\mathcal{M})\) to \(L_p(\mathcal{N})\), where \( \mathcal{M}\) and \(\mathcal{N}\) are arbitrary von Neumann algebras. All cited articles require the same \(p \)-index for domain and range space.

Our idea is to derive a norm on matrices of \(L_p\)-spaces from the order structure. This norm is quite similar to the norm used in \cite{ArKr} and \cite{Pi} (see \cite{Pi} equation (1.5)). Using this norm, we can characterize decomposable maps from \(\mathcal{M}\) to \(L_q(\mathcal{N})\) where \(1 \le q \le \infty\) and from \(L_p(\mathcal{M})\) to \(L_1(\mathcal{N})\) where \(1 \le p \le \infty\). In both cases \(\mathcal{N}\) must be injective. Then we give an example of von Neumann algebras \(\mathcal{M}\) and \(\mathcal{N}\) and a linear map from \(L_p(\mathcal{M})\) to \(L_q(\mathcal{N})\) where \( 1 \le p,q < \infty, q > 2p\) which is not decomposable. Therefore this norm cannot characterize decomposable maps for all combinations of \(p\) and \(q\).

The structure of the article is as follows: In Section 2, we give a short description on non-commutative \(L_p\)-spaces in the Haagerup-Terp construction. We also show some properties of matrices of operators. In Section 3, we define our norm on matrices and derive some properties of this norm. In Section 4, we define completely order bounded maps and show that they are decomposable for some combinations of \(p\) and \(q\).

\section{NON-COMMUTATIVE \(\boldsymbol{L_p}\)-SPACES}
Short descriptions of non-commutative \(L_p\)-spaces in the Haagerup-Terp construction can be found in several publications, but the main source is still \cite{Te}. Here, we cite some basic facts from \cite{Te} we will use.
Let \(\mathcal{M}\) be a von Neumann algebra acting on a Hilbert space \(\mathcal{H}\), and let \(\varphi\) be a normal faithful semifinite weight on \(\mathcal{M}\) with modular automorphism group \(\sigma^{\varphi}_t\). Then the crossed product \(\mathcal{M} \rtimes_{\sigma^{\varphi}} \mathbb{R}\) acts on the Hilbert space \(L_2(\mathbb{R},\mathcal{H})\) and is the von Neumann algebra generated by the operators \(\pi(x)\) and \(\lambda(s)\) where
\begin{equation*}
	\pi(x)(\xi(t)) =  \sigma^{\varphi}_{-t}(x)(\xi(t)), \quad x \in \mathcal{M}, t \in \mathbb{R}, \xi \in L_2(\mathbb{R},\mathcal{H})
\end{equation*}
and
\begin{equation*}
	\lambda(s)(\xi(t)) = \xi(t-s), \quad s,t \in \mathbb{R}, \xi \in L_2(\mathbb{R},\mathcal{H}).
\end{equation*}
For \(s \in \mathbb{R}\) the let \(W(s)\) be the unitary operator on \( L_2(\mathbb{R},\mathcal{H})\) which is defined by
\begin{equation*}\
	W(s)(\xi(t)) = e^{-\textrm{i}st}\xi(t), \quad s,t \in \mathbb{R}, \xi \in L_2(\mathbb{R},\mathcal{H}).
\end{equation*}
The dual action \(\theta\) is then defined by
\begin{equation*}\
	\theta_s(x) = W(s)xW(s)^*, \quad x \in \mathcal{M} \rtimes_{\sigma^{\varphi}} \mathbb{R}, s \in \mathbb{R}.
\end{equation*}
The elements of \(\mathcal{M}\) are the fix points under \(\theta\) when \( \mathcal{M}\) is identified with \(\pi(\mathcal{M})\):
\begin{equation*}\
	\pi(\mathcal{M}) = \{ x \in \mathcal{M} \rtimes_{\sigma^{\varphi}}\mathbb{R}: \theta_s(x) = x \textrm{ for all } s \in \mathbb{R} \}.
\end{equation*}
The crossed product \(\mathcal{M} \rtimes_{\sigma^{\varphi}} \mathbb{R}\) has a unique normal faithful semifinite trace \(\tau\) which satisfies
\begin{equation*}
	\tau(\theta_s(x)) = e^{-s}\tau(x) \textrm{ for all } x \in (\mathcal{M} \rtimes_{\sigma^{\varphi}} \mathbb{R})_+, s \in \mathbb{R}.
\end{equation*}
The existence of the trace \(\tau\) allows to consider the \(\tau \)-measurable operators. These are all closed densely defined operators \(a\) affiliated with \(\mathcal{M} \rtimes_{\sigma^{\varphi}} \mathbb{R}\) which satisfy: For every \(\varepsilon \in \mathbb{R}_+\) there exists a projection \(p \in \mathcal{M} \rtimes_{\sigma^{\varphi}} \mathbb{R}\) such that \( pL_2(\mathbb{R},\mathcal{H}) \subseteq \mathcal{D}(a)\) and \(\tau(1-p) \le \varepsilon\).
A subspace \(\mathcal{D}\) of \(L_2(\mathbb{R},\mathcal{H})\) is called \( \tau\)-dense if for every \( \varepsilon > 0 \) there is a projection \(  p \in \mathcal{M} \) such that \( pL_2(\mathbb{R},\mathcal{H}) \subseteq \mathcal{D} \) and \( \tau(1-p) \le \varepsilon \). Thus the \( \tau \)-measurable operators are those with a \( \tau \)-dense domain. The closure of a \( \tau \)-measurable operator restricted to a \( \tau \)-dense subspace is unique. Therefore, for proving some property of a \( \tau \)-measurable operator it suffices to prove the property on a \( \tau \)-dense subspace which is contained in the domain of the operator.
The \( \tau \)-measurable operators form a topological \( * \)-algebra. When  two \( \tau \)-measurable operators are added or multiplied, we have to create the closure of the sum or the product which always exist and are unique. More details for \( \tau \)-measurable operators can be found in \cite{Te}, Chapter I or \cite{Ta}, Chapter IX.2.

The action \( \theta \) can be extended to all \( \tau \)-measurable operators. The space \( L_p(\mathcal{M}) \), \( 1 \le p \le \infty \), consists of all \( \tau \)-measurable operators \(a\) for which
\begin{equation*}\
	\theta_s(a) = e^{-\frac{s}{p}}(a) \textrm{ for all } s \in \mathbb{R}.
\end{equation*}
There is a linear functional \( tr: L_1(\mathcal{M}) \to \mathbb{C}\) which has positive values for positive operators. If \( a \in L_p(\mathcal{M}) \) has the polar decomposition \( a = u|a| \), then \( u \in \mathcal{M} \) and \( |a| \in L_p(\mathcal{M}) \). The norm on \( L_p(\mathcal{M}) \) is given by
\begin{equation*}\
	\|a\|_p = tr(|a|^p)^{\frac{1}{p}}, \quad a\in L_p(\mathcal{M}).
\end{equation*}
If \( \frac{1}{p} + \frac{1}{q} = \frac{1}{r}, a \in L_p(\mathcal{M}) \) and \( b \in L_q(\mathcal{M}) \) then \( ab \in L_r(\mathcal{M}) \) and
\begin{equation*}\
	\|ab\|_r \le \|a\|_p \|b\|_q.
\end{equation*}
Especially, when \( \frac{1}{p} + \frac{1}{q} = 1\); then we get for \( a \in L_p(\mathcal{M}) \) and \( b \in L_{q}(\mathcal{M}) \) 
\begin{equation*}\
	|tr(ab)| \le \|ab\|_1 \le \|a\|_p \|b\|_{q}.
\end{equation*}
Therefore, the space \( L_{q}(\mathcal{M}) \) is isometric isomorph to the dual space of \( L_p(\mathcal{M}) \). We denote this duality by
\begin{equation*}\
	<a,b> = tr(ab) = tr(ba) \textrm{ for } a \in L_p(\mathcal{M}) \textrm{ and } b \in L_{q}(\mathcal{M}).
\end{equation*}

For \( n \in \mathbb{N} \), we denote the complex \(  n \times n \) matrices by \( M_n \) with the usual trace \( Tr \). If \( a = [a_{ij}] \) is a \( n \times n \) matrix of \( \tau \)-measurable operators and each \( a_{ij} \) acts  on the Hilbert space \( \mathcal{K} = L_2(\mathbb{R}, \mathcal{H})\), then \( a \) acts on the Hilbert space \( \mathcal{K}^n \). The operator \( a \) is densely defined, has a unique closure which we denote again by \( a \), and is \( \tau \otimes Tr \)-measurable. Especially, the elements of \( L_p(\mathcal{M}) \otimes M_n \) are \( \tau \otimes Tr \)-measurable operators. For \( \frac{1}{p} + \frac{1}{p'} = 1 \), \( n \in \mathbb{N} \), \( a = [a_{ij}] \in L_p(\mathcal{M}) \otimes M_n\), and \( b = [b_{ij}] \in L_{p'}(\mathcal{M}) \otimes M_n\), we define the duality of \( a \) and \( b \) by
\begin{equation*}\
	<a,b> = \sum_{i,j = 1}^{n}<a_{ij},b_{ji}>
\end{equation*}
The positive operators in \( L_p(\mathcal{M}) \) will be denoted by \( L_p(\mathcal{M})_+ \), and the positive operators in \( L_p(\mathcal{M}) \otimes M_n \) will be denoted by \( (L_p(\mathcal{M}) \otimes M_n)_+ \).

If \(1 \le p \le \infty\), \(\mathcal{M}\) is a von Neumann algebra, and \(L_p(\mathcal{M})\) acts on the Hilbert space \(\mathcal{K}\), then \(L_p(\mathcal{M}) \otimes M_n^{op}\) acts on the Hilbert space \(\mathcal{K}^n\) by
\begin{equation*}
	\begin{bmatrix} f_{ij} \end{bmatrix} \begin{bmatrix} \xi_i \end{bmatrix} = \begin{bmatrix}
	 \sum_{i=1}^{n} f_{ij}\xi_i \end{bmatrix}
\end{equation*}

Next, we describe some relationships between matrices and their elements.

\begin{lem}\label{orderSelfAdjoint}
	Let \( \mathcal{M} \) be a von Neumann algebra acting on a Hilbert space \( \mathcal{H} \) with a normal faithful semifinite trace \( \tau \), and let \( a \) and \( b \) be self-adjoint \( \tau \)-measurable operators. Then the following are equivalent:
	
	\begin{itemize} \item[(i)] \( -a \le b \le a \).
	
	\item[(ii)] The matrix \(
	\begin{bmatrix} a & b \\ b & a \end{bmatrix}
	\) is positive.
	\end{itemize}
\end{lem} 
\begin{proof}
	We show first that (ii) implies (i): Let \( \mathcal{D}(a)  \) and \( \mathcal{D}(b) \) denote the domains of \( a \) and \( b \). Let \( \mathcal{D} = \mathcal{D}(a) \cap \mathcal{D}(b) \). Then \( \mathcal{D} \) is a \( \tau \)-dense subspace of \( \mathcal{H} \). For \( \xi, \eta \in \mathcal{D} \), we get\label{key}
	\begin{equation}\label{aAndbSelfadjoint}
		0 \le \dfrac{1}{2} \left(
		\begin{bmatrix}  a & b \\ b & a \end{bmatrix} 
		\begin{bmatrix} \xi \\ \xi \end{bmatrix}
		\Bigg \vert
		\begin{bmatrix} \xi \\ \xi \end{bmatrix}
		\right) =
		\left(  \left(a + b\right) \xi | \xi\right).
	\end{equation}
	Hence \( -a \le b \) on \( \mathcal{D} \). By replacing the vector \( \left[\begin{smallmatrix} \xi \\ \xi \end{smallmatrix}\right] \) in \eqref{aAndbSelfadjoint} with \( \left[\begin{smallmatrix} \xi \\ -\xi \end{smallmatrix}\right] \), we get \( b  \le a \).
	
	For the implication (i) \( \Rightarrow \) (ii), we assume first that \( a \) and \( b \) are bounded. Then it follows from \cite{EfRu}, Proposition 1.2.5 that
	\( \left[\begin{smallmatrix} a & b \\ b & a \\ \end{smallmatrix} \right] \) 
	is positive. Since \( \mathcal{D}(a) \cap \mathcal{D}(b) \) is \( \tau \)-dense, there is a sequence of projections \( \left( p_n\right)^{\infty}_{n = 1}\) in \( \mathcal{M} \) such that \( p_n \le p_{n+1} \) for all \( n \in \mathbb{N} \), \( \tau(1-p_n) \to 0 \) as \( n \to \infty \), and \( p_n\mathcal{H} \subseteq  \mathcal{D}(a) \cap \mathcal{D}(b) \). Hence \( p_n a p_n \) and \( p_n b p_n \) are bounded operators and for all \( n \in \mathbb{N} \)
	\begin{equation*}
		-p_nap_n \le p_nbp_n \le pap_n.
	\end{equation*}
	Therefore, for \( \xi, \eta \in p_n\mathcal{H} \), we have
	\begin{align*}
		\left(
		\begin{bmatrix} a & b \\ b & a \end{bmatrix} 
		\begin{bmatrix} \xi \\ \eta \end{bmatrix}
		\middle |
		\begin{bmatrix} \xi \\ \eta \end{bmatrix}
		\right) &=
		\left(
		\begin{bmatrix} a & b \\ b & a \end{bmatrix} 
		\begin{bmatrix} p_n\xi \\ p_n\eta \end{bmatrix}
		\middle |
		\begin{bmatrix}p_n\xi \\ p_n\eta \end{bmatrix}
		\right) \\
		&=\left(
		\begin{bmatrix} p_nap_n & p_nbp_n \\ p_nbp_n & p_nap_n \end{bmatrix} 
		\begin{bmatrix} \xi \\ \eta \end{bmatrix}
		\middle |
		\begin{bmatrix} \xi \\ \eta \end{bmatrix}
		\right) \ge 0.
	\end{align*}
	Since the union \( \bigcup\limits_{n \in \mathbb{N}} p_n\mathcal{H}\) is \( \tau \)-dense, \(\left[ 
	\begin{smallmatrix} 
	a & b \\ b & a 
	\end{smallmatrix} \right] \) is positive .
\end{proof}

For \( n \in \mathbb{N}, 1_n \) denotes the unit matrix in \( M_n \). For a \(\tau\)-measurable operator \(a\) let supp\((a)\) denote the smallest projection which fulfils \(\textrm{supp}(a)\cdot a = a \cdot\textrm{supp}(a) = a \). If \( a \in L_p(\mathcal{M}) \) then supp\( (a) \in M \) although \( a \) is \( \tau \)-measurable with respect to a larger algebra (see \cite{Te}, Proposition II.4).

\begin{lem}\label{equalWhenLessSupp}
	Let \( \mathcal{M} \) be a von Neumann algebra with a normal faithful semifinite trace \( \tau \). Let \( a, b, c_1, c_2\) be \( \tau \)-measurable operators for which holds:
	\begin{itemize}
		\item [(i)] The operators \(a\) and \(b\) are positive.
		\item [(ii)] \( \textrm{supp}(a)\cdot c_1 \cdot \textrm{supp}(b)  = c_1 \) and \( \textrm{supp}(a)\cdot c_2 \cdot \textrm{supp}(b)  = c_2 \).
		\item [(iii)] \( ac_1b = ac_2b  \).
	\end{itemize}
	Then \( c_1 = c_2 \).
\end{lem}
\begin{proof}
	By putting \(c = c_1 - c_2 \), we may assume that \(c_2 = 0\).
	First let \(a = b\). Then (ii) can be formulated as \(\textrm{supp}(a)\cdot c \cdot \textrm{supp}(a)  = c\) and (iii) as \(aca = 0\). This means that the left support and the right support of \(c\) are less than or equal to supp\((a)\). Thus \(c\) fulfils the conditions of \cite{Schm1}, Lemma 2.2 (c), and therefore  \(c = 0.\) 
	For the general case, we put \(a' = \left[\begin{smallmatrix}
	a & 0 \\ 0 & b \end{smallmatrix}\right] \) and \(c' = \left[\begin{smallmatrix} 0 & c \\ 0 & 0 \end{smallmatrix}\right].\) Then
	\begin{equation*}
		\textrm{supp}(a') = \begin{bmatrix}
		\textrm{supp}(a) & 0 \\ 0 & \textrm{supp}(b) \end{bmatrix},
	\end{equation*}
	\begin{equation*}
		 a'c'a' = \begin{bmatrix}
		a & 0 \\ 0 & b \end{bmatrix}
		\begin{bmatrix} 0 & c \\ 0 & 0 \end{bmatrix}
		\begin{bmatrix} a & 0 \\ 0 & b \end{bmatrix} =
		\begin{bmatrix} 0 & acb \\ 0 & 0 \end{bmatrix} = 0,
	\end{equation*}
	and
	\begin{equation*}
		\textrm{supp}(a')\cdot c' \cdot \textrm{supp}(a') = 
		\begin{bmatrix} 0 & \textrm{supp}(a)\cdot c \cdot \textrm{supp}(b) \\ 0 & 0 \end{bmatrix} = 0.
	\end{equation*}
	Hence, by the first part of the proof, \(c' = 0\) which implies \(c = 0.\)
\end{proof}
\begin{thm}\label{xAsBoundedOp}
	Let \( \mathcal{M} \) be a von Neumann algebra with a normal faithful semifinite weight \( \varphi \) and acting on the Hilbert space \( \mathcal{H}  \), \( 1 \le p\le \infty \), \(n \in \mathbb{N} \), \( f, g \in L_p(\mathcal{M})_+ \), and \( x \in L_p(\mathcal{M}) \otimes M_n \) such that
	\begin{equation*}
		\begin{bmatrix} 
		f \otimes 1_n & x \\ x^* & g \otimes 1_n  
		\end{bmatrix} \ge 0.
	\end{equation*}
	Then there is an operator \( y \in \mathcal{M} \otimes M_n \) such that
	\begin{equation*}
		x = (f^{\frac{1}{2}} \otimes 1_n)y(g^{\frac{1}{2}} \otimes 1_n).
	\end{equation*}
	The operator \( y \) is bounded and \( \|y\|_\infty \le 1 \).
	The operator \( y \) is unique subject to the condition 
	\((\textrm{supp}(f) \otimes 1_n)\cdot y \cdot (\textrm{supp}(g)\otimes 1_n) = y.\)
	If \( f = g \) and \( x \) is self-adjoint, then \( y \) is self-adjoint. If \( f = g \) and \(x\) is positive, then \(y\) is positive.
\end{thm}
\begin{proof}
	Let \( \mathcal{M}_1  = \mathcal{M} \rtimes_{\sigma^{\varphi}} \mathbb{R}\), \( \mathcal{K} = L_2(\mathbb{R}, \mathcal{H}) \), and \( \tau \) be the canonical trace on \( \mathcal{M}_1 \). Let \(x = [x_{ij}]\) and
	\begin{equation*}
	\mathcal{D} = \mathcal{D}(f) \cap \mathcal{D}(f^{\frac{1}{2}}) \cap \mathcal{D}(g) \cap \mathcal{D}(g^{\frac{1}{2}}) \cap \bigcap\limits_{i,j=1}^{n} \mathcal{D}(x_{ij}).
	\end{equation*}
	Then \( \mathcal{D} \) is \( \tau \)-dense in \( \mathcal{K} \) and consequently, \( \mathcal{D}^n \) is \( \tau \otimes Tr \)-dense in \( \mathcal{K}^n \). For \( \xi, \eta \in \mathcal{D}^n \), we get
	\begin{align*}
		0 &\le \left(
		\begin{bmatrix} f \otimes 1_n & x \\ x^* & g \otimes 1_n 
		\end{bmatrix} 
		\begin{bmatrix} \xi \\ \eta \end{bmatrix} 
		\Bigg \vert
		\begin{bmatrix} \xi \\ \eta \end{bmatrix}
		\right)\\
		&= ((f \otimes 1_n) \xi | \xi) + (x \eta | \xi) + (x^* \xi | \eta) + ((g \otimes 1_n) \eta | \eta).
	\end{align*}
	This implies
	\begin{equation}\label{scalarProduct}
		- 2 \textrm{Re}((x \eta | \xi)) \le ((f \otimes 1_n) \xi | \xi) + ((g \otimes 1_n) \eta | \eta).
	\end{equation}
	We replace \( \eta \) in \eqref{scalarProduct}  by \( e^{\textrm{i}t}\eta \), and choose a suitable value for \( t \in \mathbb{R} \) to get
	\begin{equation}\label{realPart}
		2  |(x \eta | \xi)| \le ((f \otimes 1_n) \xi | \xi) + ((g \otimes 1_n) \eta | \eta).
	\end{equation}
	Then we replace \( \xi \)  by \( \lambda\xi \), \( \eta \) by \( \frac{1}{\lambda}\eta \) in \eqref{realPart}, minimize over \( \lambda \in \mathbb{R}_+ \), and get
	\begin{equation}\label{squareProduct}
		|(x \eta | \xi)|^2 \le ((f \otimes 1_n) \xi | \xi) ((g \otimes 1_n) \eta | \eta).
	\end{equation}	
	So, we can define the sesquilinear form
	\begin{equation*}
	\begin{split}
		B:~ &(g^{\frac{1}{2}} \otimes 1_n) \mathcal{D}^n \times  (f^{\frac{1}{2}} \otimes 1_n) \mathcal{D}^n \to \mathbb{C} \\
		&((g^{\frac{1}{2}} \otimes 1_n)\eta, (f^{\frac{1}{2}} \otimes 1_n)\xi) \mapsto (x\eta|\xi).
	\end{split}
	\end{equation*}
	If \(\xi, \xi', \eta, \eta' \in \mathcal{D}^n \) with \( (f^{\frac{1}{2}} \otimes 1_n)\xi = (f^{\frac{1}{2}} \otimes 1_n )\xi' \) and \( (g^{\frac{1}{2}} \otimes 1_n)\eta = (g^{\frac{1}{2}} \otimes 1_n)\eta' \), we get
	\begin{align*}
		|(x\eta|\xi) - (x\eta'|\xi')| &\le |(x\eta|\xi-\xi')| + |(x(\eta-\eta')|\xi')|\\
		&\le ((f \otimes 1_n)(\xi - \xi')|\xi - \xi')^{\frac{1}{2}}
		((g \otimes 1_n)\eta|\eta)^{\frac{1}{2}}\\
		& \quad + ((f \otimes 1_n) \xi'|\xi')^{\frac{1}{2}}
		((g \otimes 1_n)(\eta - \eta')|\eta - \eta')^{\frac{1}{2}}\\
		&= 0.
	\end{align*}
	This shows that \( B \) is well defined, and by \eqref{squareProduct}, we get 
	\begin{equation*}
	\begin{split}
		&|B((g^{\frac{1}{2}} \otimes 1_n)\eta, (f^{\frac{1}{2}} \otimes 1_n)\xi)|^2 = |(x\eta|\xi)|^2 \\
		&\quad\le ((g^{\frac{1}{2}} \otimes 1_n ) \eta | (g^{\frac{1}{2}} \otimes 1_n) \eta) 
		((f^{\frac{1}{2}} \otimes 1_n)  \xi | (f^{\frac{1}{2}}\otimes 1_n) \xi).
	\end{split}
	\end{equation*}
	Thus, \( B \) can be extended to a bounded sesquilinear form 
	\begin{equation*}
		B:~(\textrm{supp}(g) \otimes 1_n)\mathcal{K}^n \times (\textrm{supp}(f) \otimes 1_n)\mathcal{K}^n \to \mathbb{C} 
	\end{equation*}
	with norm \( \|B\| \le 1 \). Next we extend B to a bounded sesquilinear form on \( \mathcal{K}^n \times \mathcal{K}^n \) by
	\( B(\eta, \xi) = B((\textrm{supp}(g) \otimes 1_n)\eta, (\textrm{supp}(f) \otimes 1_n)\xi) \textrm{ for } \xi, \eta \in \mathcal{K}^n\). Then \( B \) has still norm \( \|B\| \le 1 \). Therefore, there is an operator \( y \in \mathcal{B}(\mathcal{K}^n) \) with \( (y\eta | \xi) = B(\eta,\xi)\) and \( \|y\|_\infty = \|B\| \le 1 \). By construction, we have
	\begin{equation*}
		(\textrm{supp}(f) \otimes 1_n) \cdot y \cdot (\textrm{supp}(g) \otimes 1_n) = y .
	\end{equation*} 	
	We still have to show that \( y \in \mathcal{M} \otimes M_n \). Let \( u \in \mathcal{M}_1' \) be a unitary element of the commutant of \( \mathcal{M}_1 \), \( \xi, \eta  \in \mathcal{K}^n\). Then there are sequences \( (\xi_i)_{i=1}^{\infty} \) and \( (\eta_i)_{i=1}^{\infty} \) with \( \xi_i, \eta_i \in \mathcal{D}^n \) for all \( i \in \mathbb{N} \),
	\begin{equation*}
		(\textrm{supp}(f) \otimes 1_n)\xi = \lim_{i \rightarrow \infty} (f^{\frac{1}{2}} \otimes 1_n)  \xi_i,
	\end{equation*}
	and 
	\begin{equation*}
		(\textrm{supp}(g) \otimes 1_n)\eta = \lim_{i \to \infty} (g^{\frac{1}{2}} \otimes 1_n)  \eta_i.
	\end{equation*}
	Since \( u \) commutes with \( f \), \( g \), supp\((f)\), and supp\((g)\), we get for all \( i \in \mathbb{N}\)
	\begin{equation*}
	\begin{split}
		&((u^* \otimes 1_n)y(u \otimes 1_n)(g^{\frac{1}{2}} \otimes 1_n ) \eta_i | (f^{\frac{1}{2}} \otimes 1_n) \xi_i) \\
		&= ((u^* \otimes 1_n)(f^{\frac{1}{2}} \otimes 1_n )y 		(g^{\frac{1}{2}} \otimes 1_n) (u \otimes 1_n) \eta_i| \xi_i) \\
		&= ((u^* \otimes 1_n)x(u \otimes 1_n) \eta_i|  \xi_i) \\
		&= (x \eta_i | \xi_i) \\
		&= ((f^{\frac{1}{2}} \otimes 1_n)y(g^{\frac{1}{2}} \otimes1_n ) \eta_i| \xi_i)
	\end{split}
	\end{equation*}
	and therefore 
	\begin{equation*}
	\begin{split}
		&((u^* \otimes 1_n)y(u \otimes 1_n)\eta | \xi ) \\
		&= ((u^* \otimes 1_n) (\textrm{supp}(f) \otimes 1_n) y  
		(\textrm{supp}(g) \otimes 1_n)(u \otimes 1_n) \eta | \xi)\\
		&= ((u^* \otimes 1_n) y (u \otimes 1_n)(\textrm{supp}(g) \otimes 1_n) \eta | (\textrm{supp}(f) \otimes 1_n) \xi)\\
		&= \lim_{i  \rightarrow \infty} ((u^* \otimes 1_n)y(u \otimes 1_n)(g^{\frac{1}{2}} \otimes 1_n) \eta_i | (f^{\frac{1}{2}} \otimes 1_n) \xi_i) \\
		&= \lim_{i  \rightarrow \infty} ((f^{\frac{1}{2}} \otimes 1_n )y(g^{\frac{1}{2}} \otimes1_n) \eta_i| \xi_i) \\
		&= (y\eta|\xi).
	\end{split}
	\end{equation*}
	Every operator in \( \mathcal{M}_1^{'} \) is a finite linear combination of unitaries and \linebreak \( (\mathcal{M}_1 \otimes M_n)^{'} = \mathcal{M}_1^{'} \otimes \mathbb{C} \). Hence \( y \in \mathcal{M}_1 \otimes M_n \).	
	We still have to show that \( y \in \mathcal{M} \otimes M_n \). Let \( y = [y_{ij}] \) with \(y_{ij} \in \mathcal{M}_1 \) for \( i, j \in \{1, \dots n\}\). Let \( s \in \mathbb{R} \). Then we have 
	\begin{align*}
		f^{\frac{1}{2}} y_{ij} g^{\frac{1}{2}} &= x_{ij} = e^{\frac{s}{p}} \theta_s(x_{ij}) \\
		&= e^{\frac{s}{p}} \theta_s(f^{\frac{1}{2}}) \theta_s(y_{ij})\theta_s(g^{\frac{1}{2}}) = f^{\frac{1}{2}}\theta_s(y_{ij})g^{\frac{1}{2}}.
	\end{align*}
	Since 
	\begin{equation*}
		y_{ij} = \textrm{supp}(f) \cdot y_{ij} \cdot \textrm{supp}(g)
	\end{equation*}
	and
	\begin{equation*}
		\theta_s(y_{ij}) = \textrm{supp}(f) \cdot \theta_s(y_{ij}) \cdot \textrm{supp}(g),
	\end{equation*}
	we can apply Lemma \ref{equalWhenLessSupp} and get \( y_{ij} = \theta_s(y_{ij}) \). Hence \( y_{ij} \in \mathcal{M} \). For the uniqueness of the decomposition, suppose that there is another \( \tilde{y} \in \mathcal{M} \otimes M_n \) such that
	\begin{equation*}
		x = (f^{\frac{1}{2}} \otimes 1_n) \tilde{y} (g^{\frac{1}{2}} \otimes 1_n) 
	\end{equation*}
	and
	\begin{equation*}
		(\textrm{supp}(f) \otimes 1_n) \cdot \tilde{y} \cdot (\textrm{supp}(g) \otimes 1_n) = \tilde{y}.
	\end{equation*}
	Then we can apply Lemma \ref{equalWhenLessSupp} and get \( y = \tilde{y} \).
	Now assume that \( x \) is self-adjoint and \( f = g \). Since \( x = x^* \), we get
	\begin{equation*}
	\begin{split}
		(f^{\frac{1}{2}} \otimes 1_n) y (f^{\frac{1}{2}} \otimes 1_n ) &= x = \frac{1}{2}(x + x^*) \\
		&= (f^{\frac{1}{2}} \otimes 1_n) \frac{1}{2} (y + y^*) (f^{\frac{1}{2}} \otimes 1_n).
	\end{split}
	\end{equation*}
	Again, we apply Lemma \ref{equalWhenLessSupp}, use the uniqueness of \( y \), and get \( y = y^*. \)
	Finally, assume that \(x\) is positive. For \(\xi \in \mathcal{K}^n\) there is a sequence \((\xi_i)_{i=1}^{\infty} \) with \(\xi_i \in \mathcal{D}^n\) for all \(i \in \mathbb{N}\) and 
	\begin{equation*}
		(\textrm{supp}(f) \otimes 1_n)\xi = \lim_{i  \rightarrow \infty} (f^{\frac{1}{2}} \otimes 1_n) \xi_i.
	\end{equation*}
	Then 
	\begin{equation*}
	\begin{split}
		(y\xi|\xi) &= (y(\textrm{supp}(f) \otimes 1_n)\xi|(\textrm{supp}(f) \otimes 1_n)\xi) \\
		&= \lim_{i \rightarrow \infty} (y(f^{\frac{1}{2}} \otimes 1_n ) \xi_i|(f^{\frac{1}{2}} \otimes  1_n) \xi_i| \xi_i) \\
		&= \lim_{i \rightarrow \infty} (x\xi_i|\xi_i) \ge 0.
	\end{split}
	\end{equation*}
	Since \(\xi\) was arbitrary, \(y\) is positive.
\end{proof}
\begin{lem}\label{fGreaterX}
	Let \( \mathcal{M} \) be a von Neumann algebra, \( 1 \le p \le \infty \), \(n  \in \mathbb{N}\), and \( x \in L_p(\mathcal{M}) \otimes M_n\). Then there exists \( f \in L_p(\mathcal{M})_+ \) such that
	\begin{equation*}
		\begin{bmatrix} f \otimes 1_n & x \\ x^* & f \otimes 1_n  
		\end{bmatrix} \ge 0.
	\end{equation*}
\end{lem}
\begin{proof}
	Since every \( x \in L_p(\mathcal{M}) \otimes M_n\) is a finite linear combination of elements of the form \( y \otimes \alpha \), where \( y \in L_p(\mathcal{M})_+ \) and \( \alpha \in M_n \), it suffices to prove the statement for \( x =  y \otimes \alpha \). Let \( \|\alpha\| \) denote the usual maximum norm of a \( n \times n \) matrix acting on \( \mathbb{C}^n \). Then we get
	\begin{equation*}
	\begin{split}
		&\begin{bmatrix} 
		\|\alpha\| y \otimes 1_n & y \otimes \alpha \\ 
		y \otimes \alpha^* & \|\alpha\|y \otimes 1_n 
		\end{bmatrix} =  \\
		&\begin{bmatrix} 
		y^{\frac{1}{2}} \otimes 1_n & 0  \\ 0 & y^{\frac{1}{2}} \otimes 1_n 
		\end{bmatrix}
		\begin{bmatrix} 
		\|\alpha\| 1_n & \alpha  \\ \alpha^* & \|\alpha\| 1_n 
		\end{bmatrix}
		\begin{bmatrix} 
		y^{\frac{1}{2}} \otimes 1_n & 0  \\ 0 & y^{\frac{1}{2}} \otimes 1_n 
		\end{bmatrix} \ge 0. \qedhere
	\end{split}
	\end{equation*}
\end{proof}
\begin{lem}\label{xFiniteLinearCombination}
	Let \( \mathcal{M} \) be a von Neumann algebra, \( 1 \le p \le \infty \), \( n  \in \mathbb{N}\), and \( x \in (L_p(\mathcal{M}) \otimes M_n)_+ \). Then \( x \) is a finite linear combination of matrices of the form \( [x_i^*x_j] \) where \( x_i \in L_{2p}(\mathcal{M}) \) for \( i \in \{1, \dots, n\} \).
\end{lem}
\begin{proof}
	By Lemma \ref{fGreaterX}, there exists \( f \in L_p(\mathcal{M})_+ \) such that
	\begin{equation*}
		\begin{bmatrix} 
		f \otimes 1_n & x \\ x^* & f \otimes 1_n 
		\end{bmatrix} \ge 0.
	\end{equation*}
	By Lemma \ref{xAsBoundedOp}, there exists \( y \in (\mathcal{M} \otimes M_n)_+ \) such that
	\begin{equation*}
		x = (f^{\frac{1}{2}} \otimes 1_n)y(f^{\frac{1}{2}} \otimes 1_n) \textrm{ and } (\textrm{supp}(f) \otimes 1_n) y (\textrm{supp}(f) \otimes 1_n) = y.
	\end{equation*}
	By \cite{Ta}, Lemma 3.1, \(y\) is a finite linear combination of matrices of the form \( [y_i^*y_j] \) where \( y_i \in \mathcal{M} \) for \( i \in \{1, \dots, n\} \). Thus, putting \( x_i = y_i f^{\frac{1}{2}} \), we get the desired result.
\end{proof}
\begin{prop}\label{operatorIsPositive}
	Let \( \mathcal{M} \) be a von Neumann algebra, \( 1 \le p, p' \le \infty \), \( \frac{1}{p} + \frac{1}{p'} = 1\), \( n  \in \mathbb{N}\), and \( x \in L_p(\mathcal{M}) \otimes M_n \). Then the following are equivalent:
	
	\begin{itemize} \item[(i)] The operator \( x \) is positive.
	
	\item[(ii)] For all \( y \in (L_{p'}(\mathcal{M}) \otimes M_n)_+\), we have \( <x,y> \ge 0. \)
	\end{itemize}
\end{prop}
\begin{proof}
	We show first that (i) implies (ii): By Lemma \ref{xFiniteLinearCombination}, is suffices to show: If \( x = [x_i^*x_j] \) and \( y = [y_i^*y_j] \) where \( x_i \in L_{2p}(\mathcal{M}) \) and \(y_i \in L_{2p'}(\mathcal{M})\) for \( i \in \{1, \dots, n\} \), then \( <x,y>\ge 0. \) Now we have
	\begin{align*}
		<x,y> &= \sum_{i,j=1}^{n}tr\left(x_i^* x_j y_j^* y_i\right) = \\
		&= tr\left(\left(\sum_{i=1}^{n}x_i y_i^*\right)^* \left(\sum_{j=1}^{n}x_j y_j^*\right)\right) \ge 0.
	\end{align*}
	For the implication from (ii) to (i), we consider the case \( p = \infty \) first. Then
	\begin{equation*}
		\pi: \mathcal{M} \to
		\mathcal{B}(L_2(\mathcal{M})),~ \pi(a)\xi = a\xi, ~  a \in \mathcal{M},\xi \in L_2(\mathcal{M})
	\end{equation*}
	is a faithful \( * \)-representation and
	\begin{equation*}
		\pi \otimes Id_{M_n}: \mathcal{M} \otimes M_n \to \mathcal{B}(L_2(\mathcal{M})^n)
	\end{equation*}
	is a faithful \( * \)-representation. If \( x = [x_{ij}] \) and \( \xi_1, \dots, \xi_n \in L_2(\mathcal{M}) \), then we get
	\begin{equation*}
		\left( \left[ \pi (x_{ij})\right]
		\begin{bmatrix} \xi_1 \\ \vdots \\ \xi_n \end{bmatrix} 
		\Bigg \vert
		\begin{bmatrix} \xi_1 \\ \vdots  \\ \xi_n \end{bmatrix}
		\right) = \sum_{i,j=1}^{n}tr\left(x_{ij} \xi_j \xi_i^* \right)  = <x, \left[ \xi_i \xi_j^* \right]\ge 0.
	\end{equation*}
	Hence \( \left( \pi \otimes Id_{M_n} \right) (x) \) is positive and therefore \( x \) is positive.
	
	Now let \( 1 \le p < \infty \). If \( x \) fulfils (ii), \(x\) must be self-adjoint. By Lemma \ref{fGreaterX}, and Lemma \ref{xAsBoundedOp}, there exists \( f \in L_{2p}(\mathcal{M})_+ \) and a self-adjoint \( b = [b_{ij}] \in \mathcal{M} \otimes M_n \) such that
	\begin{equation*}
		x = (f \otimes 1_n)b(f \otimes 1_n) \textrm{ and } \textrm{supp}(f) \cdot b \cdot \textrm{supp}(f) = b.
	\end{equation*}
	If \( p = 1 \), we put \(g = 1 \) (the unit in \( \mathcal{M} \)). If \( p > 1 \), we put \( g = f^{p-1} \). Then \( fg \in L_2(\mathcal{M}) \) , and for \( y_1, \dots, y_n \in \mathcal{M} \) we get
	\begin{equation}\label{xPositive}
		0 \le <x, \left[g y_i^* y_j g\right] > = \sum_{i,j =1}^{n}<fb_{ij}f,g y_j^* y_i g> = \sum_{i,j =1}^{n}tr\left(b_{ij} f^p y_j^* y_i f^p \right)
	\end{equation}
	By \cite{Schm}, Lemma 1.1.5, \( f^p\mathcal{M} \) is dense in \( \|\cdot\|_2 \) norm in \(\textrm{supp}(f)L_2(\mathcal{M}) \). The representation 
	\begin{equation}\label{representation}
	\begin{split}
		&\pi: \textrm{supp}(f) \mathcal{M} \textrm{supp}(f) \to \mathcal{B}(\textrm{supp}(f) L_2(\mathcal{M}))\\
		&\pi(a)\xi = a\xi,~ a \in \textrm{supp}(f) \mathcal{M} \textrm{supp}(f), \xi \in \textrm{supp}(f)L_2(\mathcal{M})
	\end{split}
	\end{equation}
	is faithful. Equation \eqref{xPositive} states that \( (\pi \otimes Id_{M_n})(b) \) is positive for a dense set of \( (\textrm{supp}(f) \otimes 1_n)L_2(\mathcal{M})^n \), and hence for all elements of \( (\textrm{supp}(f) \otimes 1_n)L_2(\mathcal{M})^n \). Since the representation in \eqref{representation} is faithful, \( b \) is positive, and thus \( x  \) is positive.
\end{proof}
\section{MATRIX NORMS ON NON-COMMUTATIVE \(\boldsymbol{L_p}\)-SPACES}
In this section \( \mathcal{M} \) always denotes a von Neumann algebra without any further restrictions. For each \( n \in \mathbb{N} \), we will define a norm on \( L_p(\mathcal{M}) \otimes M_n \) and derive some properties of this norm.

\begin{defn}
	Let \( 1 \le p \le \infty \), \( n \in \mathbb{N} \), and \( x \in L_p(\mathcal{M}) \otimes M_n \). Then we define
	\begin{equation}\label{defNorm}
		\| x \|_{p,n} =  \left\{  \frac{1}{2} \left(\|f\|_p + \|g\|_p \right)
		\Big \vert f, g \in L_p(\mathcal{M})_+, 
		\begin{bmatrix} f \otimes 1_n & x \\ x^* & g \otimes 1_n 
		\end{bmatrix} \ge 0
		\right\}.
	\end{equation}
\end{defn}

\begin{rem}
	By Lemma \ref{fGreaterX}, the set on the right side  of equation \ref{defNorm} is not empty, and therefore, the infimum is well defined.
	
	If \( p = \infty \), the norm \( \|x\|_{\infty, n} \) is identical with the usual operator norm of \( x \) considered as a bounded operator on a Hilbert space.
	
	The combination of this norm definition with Theorem \ref{xAsBoundedOp} shows that this norm is quite similar to the norm used in \cite{Pi}, equation (1.5).
\end{rem}

For \( \alpha \in M_n, \|\alpha\| \) denotes the usual operator norm of \( n \times n \)-matrices being operators on \( \mathbb{C}^n \). If \( x  \in L_p(\mathcal{M}) \otimes M_n, x = \sum_{j=1}^{k} x_j \otimes \beta_j \), then \( \alpha x = \sum_{j=1}^{k} x_j \otimes \alpha\beta_j \) and \( x\alpha = \sum_{j=1}^{k} x_j \otimes \beta_j\alpha \). For \(n \in \mathbb{N}\) let
\begin{equation}\label{definitionEpsilonij}
	\varepsilon_{ij} \textrm{ be the } n \times n \textrm{ matrix with 1 at position } (i,j) \textrm{ and 0 else.}
 \end{equation}

\begin{thm}
	Let \( 1 \le p \le \infty \textrm{ and } n \in \mathbb{N} \). Then the following holds:
	
	\begin{itemize} \item[(i)] 
	If \( x, y \in L_p(\mathcal{M}) \otimes M_n\), then 
	\( \|x+y\|_{p,n} \le \|x\|_{p,n} + \|y\|_{p,n} \).
	
	\item[(ii)]
	If \( x \in L_p(\mathcal{M}) \otimes M_n\), then \( \|x^*\|_{p,n}  = \|x\|_{p,n}\).
	
	\item[(iii)]
	If \( x \in L_p(\mathcal{M}) \otimes M_n\) and \(\alpha \in M_n \), then
	\begin{equation*}
	\|\alpha x\|_{p,n} \le \|\alpha\| \|x\|_{p.n} \textrm{ and } \|x \alpha \|_{p,n} \le \|\alpha\| \|x\|_{p,n}.
	\end{equation*}
	
	\item[(iv)]
	If \( x \in L_p(\mathcal{M}) \otimes M_n\) and \( \lambda \in \mathbb{C} \), then \( \|\lambda x \|_{p,n} = |\lambda| \|x\|_{p,n} \).
	
	\item[(v)]
	If \( x = [x_{ij}] \in L_p(\mathcal{M}) \otimes M_n\), then
	\begin{equation*}
	\max \left\{ \|x_{ij}\|_{p} \ \middle|  \ 1 \le i, j \le n \right\} \le \|x\|_{p,n} \le \sum_{i,j=1}^{n} \|x_{ij}\|_p.
	\end{equation*}
	
	\item[(vi)]
	If \( x \in L_p(\mathcal{M}) \otimes M_n\) with \( \|x\|_{p,n} = 0 \), then \( x = 0 \).
	\end{itemize}
\end{thm}
\begin{proof}
	Let \(x,y \in L_p(\mathcal{M}) \otimes M_n\), and \( \varepsilon > 0 \). Then there exist \( f_1, f_2, g_1, g_2 \in L_p(\mathcal{M})_+ \) such that
	\begin{equation*}
		\begin{bmatrix}	f_1 \otimes 1_n & x \\ x^* & g_1 \otimes 1_n 
		\end{bmatrix} \ge 0,\quad	 
		\frac{1}{2}\left(\|f_1\|_p + \|g_1\|_p\right) \le \|x\|_{p,n} + \varepsilon,
	\end{equation*} and
	\begin{equation*}	
		\begin{bmatrix} f_2 \otimes 1_n & y \\ y^* & g_2 \otimes 1_n 
		\end{bmatrix} \ge 0,\quad
		\frac{1}{2}\left(\|f_2\|_p + \|g_2\|_p\right) \le \|y\|_{p,n} + \varepsilon.
	\end{equation*}
	Then we get
	\begin{equation*}
		\begin{bmatrix} 
		\left( f_1 + f_2\right) \otimes 1_n & x+y \\ 
		\left(x+y\right)^* & \left(g_1 + g_2\right) \otimes 1_n 
		\end{bmatrix} \ge 0
	\end{equation*} and
	\begin{equation*}
		\|x+y\|_{p,n} \le \frac{1}{2}\left(\|f_1 + f_2\|_p + \|g_1 + g_2\|_p\right) \le \|x\|_{p,n} + \|y\|_{p,n} + 2\varepsilon.
	\end{equation*}
	Since \(\varepsilon\) is arbitrary, (i) is proved.
	
	To prove (ii), let \(x \in L_p(\mathcal{M}) \otimes M_n\) and \( \varepsilon > 0 \). Then there exist \( f, g \in L_p(\mathcal{M})_+ \) such that
	\begin{equation*}
		\begin{bmatrix}	f \otimes 1_n & x \\ x^* & g \otimes 1_n 
		\end{bmatrix} \ge 0, \quad
		\frac{1}{2}\left(\|f\|_p + \|g\|_p\right) \le \|x\|_{p,n} + \varepsilon.
	\end{equation*}
	Then we get
	\begin{equation*}
		0 \le \begin{bmatrix} 0 & 1 \\ 1 & 0 \end{bmatrix} 
		\begin{bmatrix}	f \otimes 1_n & x \\ x^* & g \otimes 1_n 
		\end{bmatrix} 
		\begin{bmatrix}	0 & 1 \\ 1 & 0 \end{bmatrix} =
		\begin{bmatrix}	g \otimes 1_n & x^* \\ x & f \otimes 1_n 
		\end{bmatrix}.
	\end{equation*}
	Hence, we conclude \( \|x^*\|_{p,n} \le \frac{1}{2}\left(\|f\|_p + \|g\|_p\right) \le \|x\|_{p,n} + \varepsilon \). Since \( \varepsilon \) is arbitrary, we get \( \|x^*\|_{p,n} \le \|x\|_{p,n} \). Since \( \|x\|_{p,n} = \|x^{**}\|_{p,n} \le \|x^*\|_{p,n}\), (ii) is proved.
	
	Next, we prove (iii). Let  and \( \alpha \in M_n\).
	If \( \alpha = 0 \) then \( \alpha x = 0 \) and the inequality is true. So let \( \alpha \ne 0 \). For \( \varepsilon > 0 \), there exist \( f,g \in L_p(\mathcal{M})_+ \) such that
	\begin{equation*}
		\begin{bmatrix} f \otimes 1_n & x \\ x^* & g \otimes 1_n 
		\end{bmatrix} \ge 0 \textrm{ and }
		\frac{1}{2}(\|f\|_p+\|g\|_p) \le \|x\|_{p,n} + \varepsilon. 
	\end{equation*}
	Then we have for \( \lambda > 0 \)
	\begin{align*}
		0 &\le \begin{bmatrix} \frac{1}{\lambda}\alpha & 0 \\ 
		0 & \lambda 1_n \end{bmatrix}
		\begin{bmatrix} f \otimes 1_n & x \\ x^* & g \otimes 1_n 
		\end{bmatrix}
		\begin{bmatrix} \frac{1}{\lambda}\alpha^* & 0 \\ 
		0 & \lambda 1_n \end{bmatrix} \\
		&= \begin{bmatrix} 
		\frac{1}{\lambda^2}\alpha\alpha^*(f \otimes 1_n) & \alpha x \\ (\alpha x)^* & \lambda^2 g \otimes 1_n
		\end{bmatrix} \\
		&\le \begin{bmatrix} 
		\frac{1}{\lambda^2}\|\alpha\|^2 f \otimes 1_n & \alpha x \\ 
		(\alpha x)^* & \lambda^2 g \otimes 1_n
		\end{bmatrix}.
	\end{align*}
	We put \( \lambda^2 = \|\alpha\| \) and get
	\begin{equation*}
		\|\alpha x\|_{p,n} \le \frac{1}{2}\left(\|\alpha\| \|f\|_p + \|\alpha \|g\|_p \right) \le \|\alpha\|\left(\|x\|_{p,n} + \varepsilon\right).
	\end{equation*}
	Since \( \varepsilon \) was arbitrary, we get the desired result. A similar argument proves the second inequality (iii).
	
	To prove (iv), let \(x \in L_p(\mathcal{M}) \otimes M_n\) and \( \lambda \in \mathbb{C}\).  If \( \lambda = 0 \), we have
	\begin{equation*}
		\|\lambda x\|_{p,n} = 0 = |\lambda| \|x\|_{p,n}.
	\end{equation*}
	For \( \lambda \ne 0 \), we put \( \alpha = \lambda 1_n \), apply (iii), and get
	\begin{equation*}
		\|\lambda x\|_{p,n} = \|\lambda 1_n x\|_{p,n} \le |\lambda| \|x\|_{p,n} \textrm{ and}
	\end{equation*}
	\begin{equation*}
		|\lambda| \|x\|_{p,n} = |\lambda| \left\|\frac{1}{\lambda}\lambda x\right\|_{p,n} \le |\lambda| \frac{1}{|\lambda|} \|x\|_{p,n} \le \|\lambda x\|_{p,n}.
	\end{equation*}
	
	For (v), let \( x = [x_{ij}] \in L_p(\mathcal{M}) \otimes M_n\). For \(\varepsilon > 0 \), there exist \( f, g \in L_p(\mathcal{M})_+ \) such that
	\begin{equation*}
		\begin{bmatrix}	f \otimes 1_n & x \\ x^* & g \otimes 1_n 
		\end{bmatrix} \ge 0, \quad
		\frac{1}{2}\left(\|f\|_p + \|g\|_p\right) \le \|x\|_{p,n} + \varepsilon.
	\end{equation*}
	By Theorem \ref{xAsBoundedOp}, there exists \(y \in \mathcal{M} \otimes M_n\), such that \( \|y\|_{\infty} \le 1\) and \(x = (f^{\frac{1}{2}} \otimes 1_n)y(g^{\frac{1}{2}} \otimes 1_n)\). Then \(x_{ij} = f^{\frac{1}{2}}y_{ij}g^{\frac{1}{2}}\) and \(\|y_{ij}\|_{\infty}\le \|y\|_{\infty} \le 1 \) for all \(i,j \in \{1,\cdots,n\}\). Hence
	\begin{equation*}
	\begin{split}
		\|x_{ij}\|_p &= \|f^{\frac{1}{2}}y_{ij}g^{\frac{1}{2}}\|_p \le \|f^{\frac{1}{2}}\|_{2p}\|g^{\frac{1}{2}}\|_{2p} = \sqrt{\|f\|_p\|g\|_p}\\
		&\le \frac{1}{2}\left(\|f\|_p + \|g\|_p\right) \le \|x\|_{p,n} + \varepsilon, \quad i,j \in \{1,\cdots,n\}.
	\end{split}
	\end{equation*}
	For the second inequality of (v), let \( y \in L_p({\mathcal{M}}) \) with polar decomposition \( y = v|y|\). Then \(|y^*|v = v|y|, v^*|y^*|v = |y|\), and
	\begin{equation}\label{absoluteValueGreaterThanX}
		0 \le \begin{bmatrix} 1 & 0 \\ v^* & 0 \end{bmatrix}
		\begin{bmatrix}	|y^*| & 0 \\ 0 & |y^*| \end{bmatrix}
		\begin{bmatrix}	1 & v \\ 0 & 0 \end{bmatrix} =
		\begin{bmatrix}	|y^*| & y \\ y^* & |y| \end{bmatrix}.
	\end{equation} 
	Now it follows that 
	\begin{equation*}
		\begin{bmatrix} 
		|y^*| \otimes 1_n & y \otimes 1_n \\ 
		y^* \otimes 1_n & |y| \otimes 1_n
		\end{bmatrix} \ge 0
	\end{equation*}
	and therefore \( \|y \otimes 1_n \|_{p,n} \le \|y\|_{p}\). For \( i \in \{1, \dots, n\}  \), let \( \varepsilon_{ij} \) be as in \ref{definitionEpsilonij}. Then we get for \( x = [x_{ij}] \in L_p(\mathcal{M}) \otimes M_n\)
	\begin{equation*}
		\|x\|_{p,n} = \left\| \sum_{i,j=1}^{n}x_{ij} \otimes \varepsilon_{ij}\right\|_{p,n} \le \sum_{i,j=1}^{n}\|x_{ij} \otimes 1_n \|_{p,n} \|\varepsilon_{ij}\| \le \sum_{i,j=1}^{n}\|x_{ij}\|_p.
	\end{equation*}
	
	To prove (vi), let \( \|x\|_{p,n} = 0 \). From (v), it follows that \( x_{ij} = 0 \) for all \( i,j \in \{1, \dots, n\} \). Hence \( x = 0 \).
\end{proof}

The next theorem shows that the infimum in Definition \ref{defNorm} is actually a minimum.

\begin{thm}\label{infIsMinInNorm}
	Let \( 1 \le p \le \infty \), \( n \in \mathbb{N} \), and \( x \in L_p(\mathcal{M}) \otimes M_n \). Then there exist \( f, g \in L_p(\mathcal{M})_+ \) such that
	\begin{equation*}
		\begin{bmatrix} 
		f \otimes 1_n & x \\ x^* & g \otimes 1_n 
		\end{bmatrix} \ge 0 \textrm{ and }
		\|f\|_p = \|g\|_p = \|x\|_{p,n}.
	\end{equation*}
\end{thm}
\begin{proof}
	If \( x = 0 \), we can take \( f = g = 0 \). So suppose that \( x \ne 0 \). Let \( 1 \le q \le \infty \) and \( \frac{1}{p} + \frac{1}{q} = 1\). Let \( L_q(\mathcal{M})^* \) be the dual space of \( L_q(\mathcal{M}) \). Note that \( L_q(\mathcal{M})^* \) is \( L_p(\mathcal{M}) \) when \( q < \infty \) and \( \mathcal{M}^* \) when \( q = \infty \). For \( \varepsilon > 0 \) we define
	\begin{align*}
		K_\varepsilon = \left\{ (f,g) \in L_q(\mathcal{M})_+^* \times L_q(\mathcal{M})_+^* \middle | 
		\begin{array}{l}
		\begin{bmatrix}	f \otimes 1_n & x \\ x^* & g \otimes 1_n 
		\end{bmatrix} \ge 0, \\
		\vphantom{\begin{array}{l}1 \\ 1 \end{array}}   
		\|f\| \le \|x\|_{p,n} + \varepsilon ,\\
		\|g\| \le \|x\|_{p,n} + \varepsilon \end{array} \right\}.
	\end{align*}
	The symbol \( \|\cdot\| \) denotes the norm of \( L_q(\mathcal{M})^* \). 
	The sets \( K_\varepsilon \textrm{, } \varepsilon > 0 \), have the following properties:
	
	Each \( K_\varepsilon \ne \emptyset: \) By definition of \( \|x\|_{p,n} \), there exist \( 0 \ne f,g \in L_p{\mathcal{M}}_+ \) with
	\begin{equation*}
		\begin{bmatrix} f \otimes 1_n & x \\ x^* & g \otimes 1_n \end{bmatrix} \ge 0 \textrm{ and }
		\frac{1}{2}(\|f\|_p + \|g\|_p) < \|x\|_{p,n} + \varepsilon.
	\end{equation*}
	For \( \lambda > 0 \), we get
	\begin{equation*}
		0 \le \begin{bmatrix} \lambda  & 0 \\ 0 & \frac{1}{\lambda} \end{bmatrix}
		\begin{bmatrix}	f \otimes 1_n & x \\ x^* & g \otimes 1_n \end{bmatrix}
		\begin{bmatrix} \lambda  & 0 \\ 0 & \frac{1}{\lambda} 
		\end{bmatrix} = 
		\begin{bmatrix}	\lambda^2f \otimes 1_n & x \\ x^* & \frac{1}{\lambda^2}g \otimes 1_n \end{bmatrix}.
	\end{equation*}
	We put \( \lambda = \sqrt{\frac{\|g\|_p}{\|f\|_p}} \), \( f' = \lambda f \), and \( g' = \frac{1}{\lambda}g \). Then
	\begin{equation*}
		\|f'\|_p = \|g'\|_p =  \sqrt{\|f\|_p\|g\|_p} \le \frac{1}{2}(\|f\|_p+ \|g\|_p) < \|x\|_{p,n} + \varepsilon
	\end{equation*} and 
	\begin{equation*}
		\begin{bmatrix} f' \otimes 1_n & x \\ x^* & g' \otimes 1_n 
		\end{bmatrix} \ge 0.
	\end{equation*}
	This shows that \( (f',g') \in K_\varepsilon \). Next we show that \( K_\varepsilon \) is weak\( ^* \)-closed for every \( \varepsilon > 0 \). We fix \( \varepsilon > 0 \) and \( (f,g) \) as an element of the weak\( ^* \) closure of \( K_\varepsilon \). 
	For \( a = \left[ \begin{smallmatrix} a_{11} & a_{12} \\ a_{21} & a_{22} \end{smallmatrix} \right] \in (L_q(\mathcal{M}) \otimes M_{2n})_+ \) there exist sequences \( (f_m)_{m=1}^\infty \) and \( (g_m)_{m=1}^\infty \) in \( K_\varepsilon \) such that
	\begin{equation*}
		<f \otimes 1_n , a_{11} > = \lim_{m \rightarrow \infty} <f_m \otimes 1_n, a_{11}>
	\end{equation*}
	and
	\begin{equation*}
		<g \otimes 1_n , a_{22} > = \lim_{m \rightarrow \infty} <g_m \otimes 1_n, a_{22}>,
	\end{equation*}
	and \((f_m,g_m) \in K_\varepsilon\) for all \(m \in \mathbb{N}.\) Hence, we get
	\begin{equation*}
		\left< \begin{bmatrix} f \otimes 1_n & x \\ x^* & g \otimes 1_n \end{bmatrix} \Bigg \vert a \right> = 
		\lim_{m \rightarrow \infty} 
		\left< \begin{bmatrix} f_m \otimes 1_n & x \\ x^* & g_m \otimes 1_n \end{bmatrix} \Bigg \vert a \right> \ge 0.	
	\end{equation*}
	Since this holds for every \( a \in (L_q(\mathcal{M}) \otimes M_{2n})_+ \), we can apply Proposition \ref{operatorIsPositive}, and conclude that 
	\begin{equation*}
		\begin{bmatrix}	f \otimes 1_n & x \\ x^* & g \otimes 1_n 
		\end{bmatrix} \ge 0.
	\end{equation*}
	Especially, \( f \) and \( g \) are positive. Taking \( a \in L_q(\mathcal{M}) \) with \( \|a\|_q \le 1 \), we can find a sequence \( (f_m)_{m=1}^\infty \) in \( K_\varepsilon \) such that
	\begin{equation*}
		<f, a> = \lim\limits_{m \to \infty}<f_m,a>
	\end{equation*}
	Hence \( |<f,a>| \le \|x\|_{p,n} + \varepsilon \) and
	\begin{equation*}
		\|f\| = sup \{|<f,a> |~ \|a\|_q \le 1\} \le \|x\|_{p,n} + \varepsilon.
	\end{equation*}
	Similarly, \( \|g\| \le \|x\|_{p,n} + \varepsilon \) and therefore \( (f,g) \in K_\varepsilon \). By the Banach-Alaoglu theorem, the unit ball of \( L_q(\mathcal{M}) \) is compact in the weak\( ^* \)-topology. Hence all sets \( K_\varepsilon \) are compact in the weak\(^*\)-topology.
	
	The sets \( K_\varepsilon \), \( \varepsilon > 0 \), have the finite intersection property: Given \( k \in \mathbb{N} \), \( \varepsilon_1, \cdots, \varepsilon_k > 0 \), we put \( \varepsilon = \min\{\varepsilon_1, \cdots, \varepsilon_k\} \) and get
	\begin{equation*}
		K_\varepsilon \subseteq \bigcap\limits_{i=1}^{k} K_{\varepsilon_i}.
	\end{equation*}
	Combining the finite intersection property and the weak-\( ^* \) compactness, we get
	\begin{equation}\label{kEpsilonNotEmpty}
		\bigcap\limits_{\varepsilon > 0} K_{\varepsilon} \ne \emptyset.
	\end{equation}
	Let \( (f,g) \) be in the set defined in equation (\ref{kEpsilonNotEmpty}). Then we have
	\begin{equation}\label{fGreaterThanX}
		\begin{bmatrix} f \otimes 1_n & x \\ x^* & g \otimes 1_n 
		\end{bmatrix} \ge 0
	\end{equation}
	If \( q < \infty \), then \( f,g \in L_p(\mathcal{M}) \). If \( q = \infty \), then there is a central projection in \( \mathcal{M}^{**} \) which works as projection from \( \mathcal{M}^* \) to \( L_1(\mathcal{M}) \). Hence we may assume that \(f,g \in L_1(\mathcal{M}) \). By construction, we have \( \|f\|_p, \|g\|_p \le \|x\|_{p,n} \). From equation (\ref{fGreaterThanX}), we get \( 2\|x\|_{p,n} \le \|f\|_p + \|g\|_p \). Combining both gives 
	\begin{equation*}
		\|f\|_p = \|x\|_{p,n} = \|g\|_p \qedhere
	\end{equation*}.
\end{proof}

\begin{thm}\label{normForXSelfAdjoint}
	Let \( 1 \le p \le \infty \), \( n \in \mathbb{N} \), and \( x = x^* \in L_p(\mathcal{M}) \otimes M_n \).
	\begin{itemize}
		\item[(i)]  We have \( \|x\|_{p,n} = \inf \left\{ \|f\|_p  | f \in L_p(\mathcal{M})_+, f \otimes 1_n  \pm x \ge 0 \right\} \).
		\item[(ii)] There exists \( f \in L_p(\mathcal{M})_+ \) such that
		\begin{equation*}
 			f \otimes 1_n  \pm x \ge 0  \textrm{ and } \|f\|_p = \|x\|_{p,n}.
		\end{equation*}
	\end{itemize}	
\end{thm}
\begin{proof}
	Let \(A = \inf \left\{ \|f\|_p | f \in L_p(\mathcal{M})_+, f \otimes 1_n  \pm x \ge 0 \right\} \). By Theorem \ref{infIsMinInNorm}, there exist \( f, g \in L_p(\mathcal{M})_+ \) such that
	\begin{equation}\label{fAndGMajorizesX}
		\|f\|_p = \|g\|_p = \|x\|_{p,n} \textrm{ and }
		\begin{bmatrix} f \otimes 1_n & x \\ x & g \otimes 1_n 
		\end{bmatrix} \ge 0.
	\end{equation}
	It follows then
	\begin{equation*}
		0 \le \begin{bmatrix} 0 & 1 \\ 1 & 0 \end{bmatrix}
		\begin{bmatrix}	f \otimes 1_n & x \\ x & g \otimes 1_n \end{bmatrix}
		\begin{bmatrix} 0 & 1 \\ 1 & 0 \end{bmatrix} = 
		\begin{bmatrix}	g \otimes 1_n & x \\ x & f \otimes 1_n \end{bmatrix}
	\end{equation*}
	Hence we conclude
	\begin{equation*}
		\begin{bmatrix} 
		\frac{1}{2}(f+g) \otimes 1_n & x \\ 
		x & \frac{1}{2}(f+g) \otimes 1_n 
		\end{bmatrix} \ge 0.
	\end{equation*}
	Then Lemma \ref{orderSelfAdjoint} implies that \( \frac{1}{2}(f+g) \otimes 1_n \pm x \ge 0 \). This shows that 
	\begin{equation*}
		A \le \|\frac{1}{2}(f+g)\|_p \le \|x\|_{p,n}.
	\end{equation*}
	For the converse direction, let \( \varepsilon > 0 \). Then there exist \( f \in L_p(\mathcal{M})_+ \) such that \( f \otimes 1_n \pm x \ge 0  \) and \( \|f\|_p \le A + \varepsilon \). It follows from Lemma \ref{orderSelfAdjoint} that
	\begin{equation*}
		\begin{bmatrix} f \otimes 1_n & x \\ x & f \otimes 1_n 
		\end{bmatrix} \ge 0.
	\end{equation*}
	Hence, we get \( \|x\|_{p,n}  \le \|f\|_p \le A + \varepsilon\). Since \(\varepsilon\) is arbitrary, we get \( \|x\|_{p,n} \le A \). this proves (i). To prove (ii), we take \( f \) and \( g \) from equation (\ref{fAndGMajorizesX}). Then we have
	\begin{equation*}
		\frac{1}{2}(f+g) \pm x \ge 0
	\end{equation*}
	and
	\begin{equation*}
		\|x\|_{p,n} \le \frac{1}{2}\|f+g\|_p \le \frac{1}{2}(\|f\|_p + \|g\|_p) = \|x\|_{p,n}.
	\end{equation*}
	This shows that \( \|\frac{1}{2}(f+g) \|_p = \|x\|_{p,n} \).
\end{proof}

\begin{thm}\label{1NormIs UsualNorm}
	Let \( 1 \le p \le \infty \)  and \( x \in L_p(\mathcal{M}) \). Then
	\begin{equation*}
		\|x\|_{p,1} = \|x\|_p.
	\end{equation*}
\end{thm}
\begin{proof}
	Let \( x = v|x| \) be the polar decomposition of \( x \). By equation (\ref{absoluteValueGreaterThanX}), we get \( \left[ \begin{smallmatrix}	|x^*| & x \\ x^* & |x| \end{smallmatrix} \right] \ge 0 \). Hence,
	\begin{equation*}
		\|x\|_{p,1} \le \frac{1}{2}(\|x\|_p + \|x^*\|_p) = \|x\|_p.
	\end{equation*}
	To prove to converse inequality, we apply Theorem \ref{infIsMinInNorm} and Theorem \ref{xAsBoundedOp} and get \( f,g \in L_p(\mathcal{M})_+, \|f\|_p = \|g\|_p = \|x\|_{p,1} \), \( y \in \mathcal{M}, \|y\|_\infty \le 1 \) and \(x = f^{\frac{1}{2}} y g^{\frac{1}{2}}.\)
	Hence
	\begin{equation*}
		\|x\|_p = \|f^{\frac{1}{2}} y g^{\frac{1}{2}}\|_p \le \|f\|_p^{\frac{1}{2}} \|y\|_\infty \|g\|_p^{\frac{1}{2}} = \|x\|_{p,1}. \qedhere
	\end{equation*}
\end{proof}

\section{COMPLETELY ORDER BOUNDED MAPS}
In this section, we define completely order bounded maps from \( L_p \) to \( L_q \) and show the decomposition of such maps for \( p= \infty, q \) arbitrary and for \( p \) arbitrary, \( q = 1 \). For \( 2p < q < \infty \) we give an example of a completely order bounded map which is not decomposable.

Throughout this Section, \( \mathcal{M} \) and \( \mathcal{N} \) are von Neumann algebras with no further restrictions unless stated explicitly. If \( 1 \le p,q \le \infty, n \in \mathbb{N} \), and \( T: L_p(\mathcal{M})  \to L_q(\mathcal{N})\) is a linear map, then 
\begin{equation*}
T_n: L_p(\mathcal{M}) \otimes M_n \to L_q(\mathcal{N}) \otimes M_n, [x_{ij}] \mapsto [T(x_{ij})].
\end{equation*}
We need the notion of decomposable maps which were introduced for \( C^* \)-algebras in \cite{Ha2} and extended to non-commutative \( L_p \)-spaces in \cite{Pi} and \cite{JuRu}. The above map \(T\) is decomposable if there exist completely positive maps \( S_1,S_2: L_p(\mathcal{M}) \to L_q(\mathcal{N})\) such that the induced map
\begin{equation}\label{Tdecomposable}
	\Phi: L_p(\mathcal{M}) \otimes M_2 \to L_q(\mathcal{N}) \otimes M_2, 
	\begin{bmatrix} x_{11} & x_{12} \\ x_{21} & x_{22}  
	\end{bmatrix} \mapsto
	\begin{bmatrix} S_1(x_{11}) & T(x_{12}) \\ T(x_{21}^*)^* & S_2(x_{22})  
	\end{bmatrix}
\end{equation}
is completely positive. The decomposable norm \( \|T\|_{dec} \) is defined by
\begin{equation*}
	\|T\|_{dec} = \inf \left\{ \max \left\{ \|S_1\|, \|S_2\| \right\} \right\}
\end{equation*}
where the infimum is taken over all completely positive maps \( S_1 \) and \( S_2 \) in (\ref{Tdecomposable}).

\begin{defn}
	Let \( 1 \le p,q \le \infty \). A linear map \( T: L_p(\mathcal{M}) \to L_q(\mathcal{N}) \) is called completely order bounded, if 
	\begin{equation*}
		\|T\|_{cob} = \sup \left\{ \|T_n(x)\|_{q,n} ~|~ x \in L_p(\mathcal{M}) \otimes M_n, \|x\|_{p,n} \le 1,~ n\in \mathbb{N} \right\} < \infty.
	\end{equation*}
\end{defn}

The name completely order bounded will be justified by Theorem \ref{orderIntervalToOrderInterval} where we show that a completely order bounded map maps order intervals to order intervals uniformly over all matrix levels.

\begin{prop}
	Let \( 1 \le p,q \le \infty \) and \( T: L_p(\mathcal{M}) \to L_q(\mathcal{N}) \)  be completely positive. Then \( T \) is completely order bounded and 
	\begin{equation*}
		\|T\|_{cob} = \|T\| = \sup \left\{ \|T(x)\|_q~|~x \in L_p(\mathcal{M})_+,~ \|x\|_p \le 1  \right\}.
	\end{equation*}
\end{prop}
Here, \( \|T\| \) means the usual operator norm of \( T \) as a bounded operator on a normed vector space.
\begin{proof}
	Let \( \Lambda = \sup \left\{ \|T(x)\|_q~|~x \in L_p(\mathcal{M})_+,~ \|x\|_p \le 1  \right\} \).
	By Theorem \ref{1NormIs UsualNorm}, we have
	\begin{equation*}
		\sup \left\{ \|T(x)\|_q~|~x \in L_p(\mathcal{M})_+,~ \|x\|_p \le 1  \right\} \le \|T\| \le \|T\|_{cob}.
	\end{equation*}
	For the opposite inequality, let \( n \in \mathbb{N}, x \in L_p(\mathcal{M}) \otimes M_n \) with \( \|x\|_{p,n} \le 1 \). By Theorem \ref{infIsMinInNorm}, there exist \( f,g \in L_q(\mathcal{N})_+ \) such that 
	\begin{equation*}
		\|f\|_p = \|g\|_p = \|x\|_{p,n} \textrm{ and }
		\begin{bmatrix}	f \otimes 1_n & x \\ x^* & g \otimes 1_n  
		\end{bmatrix} \ge 0.
	\end{equation*}
	Since \( T \) is completely positive, we have \( T_n(x^*) = T_n(x)^* \) and
	\begin{equation*}
		\begin{bmatrix}	T(f) \otimes 1_n & T_n(x) \\ T_n(x)^* & T(g) \otimes 1_n 
		\end{bmatrix} \ge 0
	\end{equation*}
	Hence,
	\begin{equation*}
		\|T_n(x)\|_{q,n} \le \frac{1}{2}(\|T(f)\|_q+\|T(g)\|_q)
		\le \Lambda \|x\|_{p,n} \le \Lambda. \qedhere
	\end{equation*}
\end{proof}
A decomposable map \(T\) is completely order bounded and \(\|T\|_{cob} \le \|T\|_{dec}\). Next, we show that the composition of completely order bounded maps is completely order bounded.

\begin{thm}
	Let \( 1 \le p_1, p_2, p_3 \le \infty, \mathcal{M}_1, \mathcal{M}_2, \mathcal{M}_3 \) be von Neumann algebras, and \( T_1: L_{p_1}(\mathcal{M}_1) \to L_{p_2}(\mathcal{M}_2), T_2: L_{p_2}(\mathcal{M}_2) \to L_{p_3}(\mathcal{M}_3)\) be completely order bounded maps. Then the composition \( T_2 \circ T_1\) is completely order bounded and 
	\begin{equation*}
		\|T_2 \circ T_1\|_{cob} \le \|T_2\|_{cob} \|T_1\|_{cob}.
	\end{equation*}
\end{thm}
\begin{proof}
	Let \( n \in \mathbb{N} \) and \( x \in L_{p_1}(\mathcal{M}_1) \otimes M_n \) such that \( \|x\|_{p_1,n} \le 1 \). Then 
	\begin{equation*}
		\|T_{2,n}(T_{1,n}(x))\|_{p_3,n} \le \|T_2\|_{cob}\|T_{1,n}(x))\|_{p_2,n} \le \|T_2\|_{cob}\|T_1\|_{cob}. \qedhere
	\end{equation*}
\end{proof}

The next theorem justifies the name completely order bounded: Completely order bounded maps map order intervals into order intervals uniformly over all matrix levels.

\begin{thm}\label{orderIntervalToOrderInterval}
	Let \( 1 \le p, q \le \infty \), \( T:L_p(\mathcal{M}) \to L_q(\mathcal{N}) \) be completely order bounded, and \( f \in L_p(\mathcal{M})_+ \). Then there exist \( g_1, g_2 \in L_q(\mathcal{N})_+ \) such that \( \|g_1\|_q, \|g_2\|_q \le \|T\|_{cob} \|f\|_p \) and for all \(n \in \mathbb{N}, x \in L_p(\mathcal{M}) \otimes M_n\) 
	\begin{equation}\label{orderBounded}
		\begin{bmatrix} f \otimes 1_n & x \\ x^* & f \otimes 1_n  
		\end{bmatrix} \ge 0 \textrm{ implies }
		\begin{bmatrix}	g_1 \otimes 1_n & T_n(x) \\ T_n(x)^* & g_2 \otimes 1_n  
		\end{bmatrix} \ge 0.
	\end{equation}
\end{thm}
\begin{proof}
	For \( q > 1, n\in \mathbb{N}, x \in L_p(\mathcal{M}) \otimes M_n \) with \( \left[ \begin{smallmatrix} f \otimes 1_n & x \\ x^* & f \otimes 1_n  \end{smallmatrix} \right] \ge 0 \), we put
	\begin{multline*}
		K(x) = \left\{ (g_1,g_2) \in L_q(\mathcal{N})_+ \times L_q(\mathcal{N})_+ \middle|  
		\begin{array}{l}
		\begin{bmatrix}	g_1 \otimes 1_n & T_n(x) \\ T_n(x)^* & g_2 \otimes 1_n 
		\end{bmatrix} \ge 0, \\
		\vphantom{\begin{array}{l}1 \\ 1 \end{array}}   \|g_1\|_q,\|g_2\|_q \le \|T\|_{cob} \|f\|_p 
		\end{array} \right\}.
	\end{multline*}
	For \( q = 1 \) we define a similar set, but take pairs \( (g_1,g_2) \in \mathcal{N}^*_+ \times \mathcal{N}^*_+ \) instead of \( L_1(\mathcal{N})_+ \times L_1(\mathcal{N})_+ \). 
	By Theorem \ref{infIsMinInNorm}, \( K(x) \) is not empty. Further, \( K(x) \) is weak\( ^* \)-closed. This is proved similarly as in the proof of Theorem \ref{infIsMinInNorm}. By the Banach-Alaoglu theorem, the unit ball of \( L_q(\mathcal{N}) \) is compact in the weak\( ^* \)-topology for \(q > 1\). The same holds for \(\mathcal{N}^*\). Hence all sets \( K(x) \) are compact in the weak \( ^* \)-topology. The sets \( K(x), x \in L_p(\mathcal{N}) \otimes M_n, n\in \mathbb{N}, \) have the finite intersection property: For \( k \in \mathbb{N}, n_1, \cdots, n_k \in \mathbb{N} \), and \( x_i \in L_p(\mathcal{M}) \otimes M_{n_i} \) with
	\begin{equation*}
		\begin{bmatrix}	f \otimes 1_{n_i} & x_{n_i} \\ 
		x_{n_i}^* & f \otimes 1_{n_i}  
		\end{bmatrix} \ge 0,~ i\in \{1, \cdots, k\},
	\end{equation*}
	we put all \( x_i \) in the diagonal matrix \( x = \textrm{diag}(x_1, \cdots, x_k) \) and set \( n = \sum_{i=1}^{k} n_i \). Then
	\begin{equation*}
		\begin{bmatrix}
		f \otimes 1_{n} & x \\ 
		x^* & f \otimes 1_{n}  
		\end{bmatrix} \ge 0.
	\end{equation*}
	By Theorem \ref{infIsMinInNorm}, there exist \(  g_1, g_2 \in L_q(\mathcal{N})_+ \) such that
	\begin{equation*}
		\begin{bmatrix}
		g_1 \otimes 1_n & T_n(x) \\ 
		T_n(x)^* & g_2 \otimes 1_n 
		\end{bmatrix} \ge 0
		\textrm{ and } \|g_1\|, \|g_2\| \le \|T\|_{cob} \|f\|_p.
	\end{equation*}
	Hence, \( (g_1, g_2) \in K(x_i) \textrm{ for all } i \in \{1, \cdots, k\}\). We conclude that 
	\begin{equation*}
		\bigcap \limits_{n \in \mathbb{N},~ x \in L_p(\mathcal{M}) \otimes M_n} K(x) \ne \emptyset.
	\end{equation*}
	We take a pair \( (g_1,g_2) \) of this set. If \( q > 1 \), this pair fulfils \eqref{orderBounded}. If \( q = 1 \) there is a central projection \( z \in \mathcal{N}^{**} \) which maps \( \mathcal{N}^* \) to \( L_1(\mathcal{N}) \). Then the pair \( (zg_1,zg_2) \) is in \(L_1(\mathcal{N})_+ \times L_1(\mathcal{N})_+ \) and fulfils \eqref{orderBounded}.
\end{proof}

\begin{thm}\label{mapFactoredThroughLInfinity}
	Let \(1 \le q \le \infty\) and let \( T: \mathcal{M} \to L_q(\mathcal{N}) \) be completely order bounded. Then there exist \( f,g \in L_q(\mathcal{N})_+ \) and a completely order bounded map \({ S : \mathcal{M} \to \mathcal{N}} \) such that
	\begin{equation*}
	T(x) = f^{\frac{1}{2}}S(x)g^{\frac{1}{2}} \textrm{ for all } x \in \mathcal{M}
	\end{equation*}
	and 
	\begin{equation*}
	\|f\|_q, \|g\|_q \le \|T\|_{cob},~ \|S\|_{cob} \le 1.
	\end{equation*}
\end{thm}
Note that in case of a linear map from \( \mathcal{M} \) to \( \mathcal{N} \) completely order bounded is identical to completely bounded.
\begin{proof}
	By Theorem \ref{orderIntervalToOrderInterval}, there exist \( f, g \in L_q(\mathcal{N})\), such that \( \|f\|_q, \|g\|_q \le \|T\|_{cob}\) and
	\begin{equation*}
	\left[ \begin{array}{cc}
	f \otimes 1_n & T_n(x) \\ 
	T_n(x)^* & g \otimes 1_n  
	\end{array} \right] \ge 0 \textrm{ for all } n\in \mathbb{N}, x \in \mathcal{M} \otimes M_n, \|x\|_{\infty,n} \le 1.
	\end{equation*}
	Let \(x \in \mathcal{M}\). According to Theorem \ref{xAsBoundedOp}, there is an unique \(y \in \mathcal{N}, \|y\|_{\infty} \le 1\) such that \(T(x) = f^{\frac{1}{2}}yg^{\frac{1}{2}}\) and \(\textrm{supp}(f)\cdot y \cdot \textrm{supp}(g) = y\). We put \(S(x) = y\). Then \(S\) is a map from \(\mathcal{M}\) to \(\mathcal{N}\).
	If \(x, x_1, x_2 \in\mathcal{N}, x = x_1 + x_2,\) then there are  \(y, y_1, y_2 \in \mathcal{N}\) such that \(T(x) =  f^{\frac{1}{2}}yg^{\frac{1}{2}}, \textrm{supp}(f) \cdot y \cdot \textrm{supp}(g) = y\) and \(T(x_i) = f^{\frac{1}{2}}y_ig^{\frac{1}{2}}, \textrm{supp}(f) \cdot y_i \cdot \textrm{supp}(g) = y_i, i=1,2\). Then we get
	\begin{equation*}
		f^{\frac{1}{2}}yg^{\frac{1}{2}} = T(x) = T(x_1) + T(x_2) = f^{\frac{1}{2}}(y_1+y_2)(g^{\frac{1}{2}}),
	\end{equation*}
	\begin{equation*}
		\textrm{supp}(f) \cdot (y_1+y_2) \cdot \textrm{supp}(g) = y_1 + y_2.
	\end{equation*}		
	Hence we conclude from Lemma \ref{equalWhenLessSupp} that \(y = y_1 + y_2\) which means that \(S\) is additive. Similarly, we show that \(S(\lambda x) = \lambda S(x) \textrm{ for } \lambda \in \mathbb{C}, x \in \mathcal{M}\). Now for \( n \in \mathbb{N}, \) let \( x \in \mathcal{M} \otimes M_n\) with \(\|x\|_{\infty,n} \le 1\). According Theorem \ref{xAsBoundedOp}, there is \(y \in \mathcal{N} \otimes M_n\) such that \((\textrm{supp}(f) \otimes 1_n) \cdot y \cdot (\textrm{supp}(g) \otimes 1_n) = y\) and \(T_n(x) = (f^{\frac{1}{2}} \otimes 1_n)y(g^{\frac{1}{2}} \otimes 1_n) \). Then, we have
	\begin{equation*}
		(f^{\frac{1}{2}} \otimes 1_n)S_n(x)(g^{\frac{1}{2}} \otimes 1_n)  = T_n(x) = (f^{\frac{1}{2}} \otimes 1_n)y(g^{\frac{1}{2}} \otimes 1_n).
	\end{equation*}
	 Hence, we conclude that \(S_n(x) = y\) which shows that \(\|S_n(x)\|_{\infty,n} \le 1\). So \(\|S\|_{cob} \le 1\). 
\end{proof}
\begin{thm}
	Let \(\mathcal{N}\) be injective, \( 1 \le q \le \infty \) and let \( T: \mathcal{M} \to L_q(\mathcal{N}) \) be completely order bounded. Then there exist linear maps \(T_i: \mathcal{M} \rightarrow L_q(\mathcal{N})\) such that the map 
	\begin{equation*}
	\begin{split}
	\Phi: &\mathcal{M} \otimes M_2 \to L_q(\mathcal{N}) \otimes M_2 \\
	&\begin{bmatrix} 
	x_{11} & x_{12} \\ 	x_{21} & x_{22}  
	\end{bmatrix} \mapsto
	\begin{bmatrix} 
	T_1(x_{11}) & T(x_{12}) \\ 
	T(x_{21}^*)^* & T_2(x_{22})  
	\end{bmatrix}
	\end{split}
	\end{equation*}
	is complete positive and \(\|T_1\|, \|T_2\| \le \|T\|_{cob} \). 
	Thus \(T\) is decomposable and \(\|T\|_{dec} = \|T\|_{cob}\).
\end{thm}
\begin{proof}
	By Theorem \ref{mapFactoredThroughLInfinity}, there exist \(f, g \in L_q(\mathcal{N})_+,  \|f\|_q, \|g\|_q \le \|T\|_{cob}\) and \(S: \mathcal{M} \rightarrow \mathcal{N}, \|S\|_{cob} \le 1\) such that
	\begin{equation*}
	T(x) = f^{\frac{1}{2}}S(x)g^{\frac{1}{2}}, \quad x \in \mathcal{M}.
	\end{equation*}
	Since for a linear map from \(\mathcal{M}\) to \(\mathcal{N}\) completely bounded is the same as completely order bounded, we can apply \cite{Wi}, Theorem 4.5, and get linear maps \(S_1, S_2: \mathcal{M} \rightarrow \mathcal{N}\) such that \(\|S_i\| \le 1, i = 1,2\), and the map
	\begin{equation*}
	\left[ \begin{array}{cc} 
	x_{11} & x_{12} \\ 
	x_{21} & x_{22}  
	\end{array} \right] \mapsto
	\left[ \begin{array}{cc} 
	S_1(x_{11}) & S(x_{12}) \\ 
	S(x_{21}^*)^* & S_2(x_{22})  
	\end{array} \right]
	\end{equation*}
	is completely positive. Then the linear maps \(T_i: \mathcal{M} \rightarrow L_q(\mathcal{N})\) where \(T_1(x) = f^{\frac{1}{2}}S_1(x)f^{\frac{1}{2}}, T_2(x) = g^{\frac{1}{2}}S_2(x)g^{\frac{1}{2}}, x \in \mathcal{M},\) have the desired properties.
\end{proof}
\begin{rem}
	Theorem 4.5 in \cite{Wi} states the decomposition for self-adjoint maps. The proof of Proposition 1.3 in \cite{Ha2} shows that this is also true for maps which are not self-adjoint.
\end{rem}
The next goal is to prove that completely order bounded maps from \(L_p(\mathcal{M})\) to \(L_1(\mathcal{N}).\) are decomposable. This proof is divided into several steps.

\begin{lem}\label{selfadjointIsTensorOfSelfadjoints}
	Let \( k \in \mathbb{N}\) and \(a \in L_p(\mathcal{M}) \otimes M_2 \otimes M_{2k}\) be self- adjoint. Then \(a\) can be written in the form
	\begin{equation*}
	a = \sum_{i=1}^{n} a_i \otimes \alpha_i
	\end{equation*}
	where all \(a_i \in L_p(\mathcal{M}) \otimes M_2\) and all \(\alpha_i  \in M_{2k}\) are self-adjoint.
\end{lem}
\begin{proof}
	By Lemma \ref{fGreaterThanX}, there exists \(f \in L_p(\mathcal{M})_+\) such that 
	\begin{equation*}
		\left[ \begin{array}{cc} 
		f \otimes 1_2 \otimes 1_{2k} & a \\ 
		a^* & f \otimes 1_2 \otimes 1_{2k}  
		\end{array} \right] \ge 0.
	\end{equation*}
	By Lemma \ref{xAsBoundedOp}, there is \(y \in \mathcal{M} \otimes M_2 \otimes M_{2k}\) such that \(y\) is self-adjoint and \(a = (f^{\frac{1}{2}} \otimes 1_2 \otimes 1_{2k})y(f^{\frac{1}{2}} \otimes 1_2 \otimes 1_{2k})\). By \cite{Ta}, Lemma IV.4.4, \(y\) can be written in the form
	\begin{equation*}
		y = \sum_{i=1}^{n} y_i \otimes \alpha_i,
	\end{equation*}
	where all \(y_i \in L_p(\mathcal{M}) \otimes M_2\) and all \(\alpha_i \in M_{2k}\) are self-adjoint. Then 
	\begin{equation*}
	a = \sum_{i=1}^{n} (f^{\frac{1}{2}} \otimes 1_2 )y_i(f^{\frac{1}{2}} \otimes 1_2 ) \otimes \alpha_i
	\end{equation*}
	gives the desired decomposition.
\end{proof}
Let \(Tr\) be the usual trace on \(M_{2k}\). We define the duality between \(L_1(M_k) \otimes M_2\) and \(M_{2k}\) by \(<a,b> = Tr(ab), a \in L_1(M_k) \otimes M_2, b \in M_{2k}.\) If \(n \in \mathbb{N}, a = [a_{ij}], b = [b_{ij}],\) where \(a_{ij} \in L_1(M_k) \otimes M_2\) and \( b_{ij} \in M_k \otimes M_2 \) then \(<[a_{ij}],[b_{ij}]> = \sum_{i,j=1}^{2k} <a_{ij}, b_{ji}>.\)
We define the linear functional
\begin{equation}\label{defi
	neOmega}
	\begin{split}
	\omega: &L_1(M_k) \otimes M_2 \otimes M_{2k} \rightarrow \mathbb{C} \\
	&\sum_{i=1}^{n} a_i \otimes b_i \mapsto \sum_{i=1}^{n} <a_i,b_i^t> 
	\end{split}
\end{equation}
where \(b_i^t\) denotes the transposed matrix of \(b_i.\)
\begin{lem}\label{omegaIsPositive} Let \(a \in (L_1(M_k) \otimes M_2 \otimes M_{2k})_+.\) Then
	\(\omega(a) \ge 0\).
\end{lem}
\begin{proof}
	Let \( a^\frac{1}{2} = \sum_{i=1}^{n} a_i \otimes \alpha_i.\)
	Then \(a^\frac{1}{2}\) is self-adjoint, \(a = \sum_{i,j=1}^{n} a_i^*a_j \otimes \alpha_i^* \alpha_j,\)
	and 
	\begin{equation*}
	\begin{split}
		\omega(a) &= \sum_{i,j=1}^{n} \omega(a_i^*a_j \otimes \alpha_i^* \alpha_j) = \sum_{i,j=1}^{n} \left<a_i^*a_j,\alpha_j^t\alpha_i^{*t}\right> \\
		&= \left<\left[a_i^*a_j\right], \left[\alpha_i^t \alpha_j^{*t} \right]\right>
	\end{split}
	\end{equation*}
	Since \(\left[a_i^*a_j\right]\) and \(\left[\alpha_i^t \alpha_j^{*t} \right]\) are positive matrices, the last expression is positive by Proposition \ref{operatorIsPositive}.
\end{proof}
For a linear map \(T: L_p(\mathcal{M}) \rightarrow L_1(M_k)\) we define 
\begin{equation}\begin{split}\label{defineTtilde}
	\tilde{T}: &L_p(\mathcal{M}) \otimes M_2 \rightarrow L_1(M_k) \otimes M_2\\
	& \begin{bmatrix} 
	x_{11} & x_{12} \\ x_{21} & x_{22} 
	\end{bmatrix} \mapsto
	\begin{bmatrix} 
	0 & T(x_{12}) \\ T(x_{21}^*)^* & 0 
	\end{bmatrix}.
\end{split}\end{equation}
Then \(\tilde{T}\) is a self-adjoint linear map. We define the linear functional
\begin{equation}\label{definePhiT}
\begin{split}
	\varphi_T: &L_p(\mathcal{M}) \otimes M_2 \otimes M_{2k} \rightarrow \mathbb{C}\\
	& x \mapsto \omega \circ \tilde{T}_{2k}(x).
\end{split}
\end{equation}
\begin{lem}\label{tCobDominatesPhiAMinusB}
	Let \(T: L_p(\mathcal{M}) \rightarrow L_1(M_k)\) be completely order bounded and \(\varphi_T\) as in \eqref{definePhiT}. Let \(a, b \in (L_p(\mathcal{M}) \otimes M_2 \otimes M_{2k})_+, f, g \in L_p(\mathcal{M})_+\) such that
	\begin{equation}\label{aPlusbLessThanf}
		0 \le a + b \le f \otimes \varepsilon_{11} \otimes 1_{2k} + g \otimes \varepsilon_{22} \otimes 1_{2k}.
	\end{equation}
	Then
	\begin{equation*}
		\|T\|_{cob}( \|f\|_p + \|g\|_p ) + \varphi_T ( a - b) \ge 0.
	\end{equation*}
\end{lem}
Here \(\varepsilon_{ij}, i,j = 1,2\) are the \(2 \times 2\) matrices defined in \eqref{definitionEpsilonij}.
\begin{proof}
	Since \(a-b\) is self-adjoint, by Lemma \ref{selfadjointIsTensorOfSelfadjoints}, we can write
	\begin{equation*}
		a-b = \sum_{i=1}^{n} a_i \otimes \alpha_i
	\end{equation*}
	where \(a_i \in L_p(\mathcal{M}) \otimes M_2\) and \(\alpha_i,  \in M_{2k}\) are all self-adjoint. For \(i \in \{1, \cdots,n\}\)  let 
	\begin{equation*}
		a_i = \begin{bmatrix} 
		a_{i,11} & a_{i,12} \\ a_{i,21} & a_{i,22}   
		\end{bmatrix},
		\alpha_i = \begin{bmatrix} 
		\alpha_{i,11} & \alpha_{i,12} \\ \alpha_{i,21} & \alpha_{i,22}   
		\end{bmatrix},
	\end{equation*}
	where \(a_{i,st} \in L_p(\mathcal{M})\) and \(\alpha_{i,st} \in M_k\) for \(s,t = 1,2\).
	Since each \(a_i\) is self-adjoint, we have \(a_{i,21}^* = a_{i,12}\), \(\alpha_{i,21}^* = \alpha_{i,12}\), and \(<T(a_{i,21}^*)^*,\alpha_{i,21}^t> = \linebreak<T(a_{i,12})^*,\alpha_{i,12}^{*t}> = \overline{<T(a_{i,12}),\alpha_{i,12}^t>}\) for all \(i \in \{1, \cdots,n\}\). Therefore
	\begin{equation*}
	\begin{split}
		\varphi_T(a-b) &= \sum_{i=1}^{n} \varphi_T( a_i \otimes \alpha_i) \\
		&= \sum_{i=1}^{n} \left< \begin{bmatrix}
		0 & T(a_{i,12}) \\ T(a_{i,21}^*)^* & 0
		\end{bmatrix}, \begin{bmatrix}
		\alpha_{i,11}^t & \alpha_{i,21}^t \\
		\alpha_{i,12}^t & \alpha_{i,22}^t
		\end{bmatrix} \right> \\
		&= \sum_{i=1}^{n} \left(<T(a_{i,12}),\alpha_{i,12}^t> + <T(a_{i,12})^*, \alpha_{i,12}^{*t}> \right).
	\end{split}
	\end{equation*}
	By Proposition \ref{orderSelfAdjoint}, we have 
	\begin{equation}\label{a minus b majorized by f and g)}
		\begin{bmatrix}
		(f \otimes \varepsilon_{11} + g \otimes \varepsilon_{22}) \otimes 1_{2k} & a-b \\
		a-b & (f \otimes \varepsilon_{11} + g \otimes \varepsilon_{22}) \otimes 1_{2k}
	\end{bmatrix} \ge 0.
	\end{equation}
	We multiply the matrix in \eqref{a minus b majorized by f and g)} from left with the matrix \(\gamma\) and from right with the transposed matrix \(\gamma^t\), where \(\gamma\) is the \(8k \times 4k\) matrix \(\left[\begin{smallmatrix}
		\varepsilon_{11} \otimes 1_{2k} \\ \varepsilon_{22} \otimes 1_{2k}
	\end{smallmatrix}\right]\) and get
	\begin{equation}\label{extractai}
	\begin{split}
		0 &\le \varepsilon_{11}(f \otimes \varepsilon_{11} + g \otimes \varepsilon_{22}) \varepsilon_{11} \otimes 1_{2k} + (\varepsilon_{11} \otimes1_{2k}(a-b)(\varepsilon_{22} \otimes1_{2k}) \\
		&\quad + (\varepsilon_{22} \otimes1_{2k})(a-b)(\varepsilon_{11} \otimes1_{2k}) + \varepsilon_{22}(f \otimes \varepsilon_{11} + g \otimes \varepsilon_{22}) \varepsilon_{22} \otimes 1_{2k} \\
		&= f \otimes \varepsilon_{11} \otimes 1_{2k} + \sum_{i=1}^{n} a_{i,12} \otimes \varepsilon_{12} \otimes \alpha_i \\
		&\quad +\sum_{i=1}^{n} a_{i,21} \otimes \varepsilon_{21} \otimes \alpha_i + g \otimes \varepsilon_{22} \otimes 1_{2k}.
	\end{split}
	\end{equation}
	This inequality can be written as \(\|\sum_{i=1}^{n} a_{i,12} \otimes \alpha_i\|_{p,2k} \le \frac{1}{2}(\|f\|_p + \|g\|_p). \)
	Hence \(\|T(\sum_{i=1}^{n} a_{i,12} \otimes \alpha_i)\|_{1,2k} \le \frac{1}{2} \|T\|_{cob}(\|f\|_p + \|g\|_p)\). By Theorem \ref{infIsMinInNorm}, there exist \(f_1, g_1 \in L_1(M_k)_+\) such that \( \|f_1\|_1 = \|g_1\|_1 \le  \frac{1}{2} \|T\|_{cob}(\|f\|_p + \|g\|_p)\) and 
	\begin{equation}\label{TaiTensorAlphi}
		\begin{bmatrix}
			f_1 \otimes 1_{2k} & \sum_{i=1}^{n} T(a_{i,12}) \otimes \alpha_i \\
			T(a_{i,12})^* \otimes \alpha_i & g_1 \otimes 1_{2k}
		\end{bmatrix} \ge 0.
	\end{equation}
	We apply \(\omega\) to \eqref{TaiTensorAlphi} and get
	\begin{equation*}
	\begin{split}
		0 &\le <f_1 \otimes \varepsilon_{11}, 1_{2k}> + <g_1 \otimes \varepsilon_{22}, 1_{2k}> \\
		&\quad + \sum_{i=1}^{n} \left(<T(a_{i,12}) \otimes \varepsilon_{12}, \alpha_i^t> + <T(a_{i,12})^* \otimes \varepsilon_{21},\alpha_i^t>\right) \\
		&= \|f_1\|_1 + \|g_1\|_1 + \sum_{i=1}^{n} (<T(a_{i,12}), \alpha_{i,12}^t> + <T(a_{i,12})^*, \alpha_{i,12}^{*t}>) \\
		&= \|f_1\|_1 + \|g_1\|_1 + \varphi_T(a-b).
	\end{split}
	\end{equation*}
	Since \(\|f_1\|_1 + \|g_1\|_1 \le \|T\|_{cob}(\|f\|_p + \|g\|_p),\) the proof is finished.
	\end{proof}
\begin{lem}\label{xijTensoryjiPositive}
	Let \(x \in (L_p(\mathcal{M}) \otimes M_2 \otimes  M_n)_+\) and \(y  \in (M_k \otimes M_2 \otimes M_n)_+\) where \(x= [x_{ij}]\) with \(x_{ij} \in L_p(\mathcal{M} \otimes M_2\) and \(y = [y_{ij}]\) with \(y_{ij} \in M_k \otimes M_2\) for \(i,j \in \{1,\cdots,n\}\) Then
	\begin{equation*}
		\sum_{i,j = 1}^{n} x_{ij} \otimes y_{ji}^t \ge 0.
	\end{equation*}
\end{lem}
\begin{proof}
	By Lemma \ref{xAsBoundedOp}, there is \(f \in L_p(\mathcal{M})_+\) and \(a \in (\mathcal{M} \otimes M_2 \otimes M_n)_+\) such that \(x = (f^{\frac{1}{2}} \otimes 1_{2n})a(f^{\frac{1}{2}} \otimes 1_{2n})\). By \cite{Ta}, Lemma IV.3.1, \(a\) can be written as a finite sum of matrices of the form \([a_i^*a_j]\) where \(a_1, \cdots, a_n \in \mathcal{M} \otimes M_2\). Similarly, \(y\) can be written as a finite sum of matrices of the form \([y_i^*y_j]\) where \(y_1, \cdots, y_n \in M_k \otimes M_2\). Thus 
	\(\sum_{i,j = 1}^{n} x_{ij} \otimes y_{ij}^t\) is a finite sum of elements of the form \(\sum_{i,j = 1}^{n} (f^{\frac{1}{2}} \otimes 1_2) a_i^*a_j (f^{\frac{1}{2}} \otimes 1_2) \otimes (y_i^*y_j)^t\). Then we have
	\begin{equation*}
	\begin{split}
	&\sum_{i,j = 1}^{n} (f^{\frac{1}{2}} \otimes 1_2) a_i^*a_j (f^{\frac{1}{2}} \otimes 1_2) \otimes (y_j^*y_i)^t \\  =&\sum_{i,j=1}^{n} (f^{\frac{1}{2}} \otimes 1_2)a_i^*a_j (f^{\frac{1}{2}} \otimes 1_2) \otimes y_i^{*t}y_j^t \\
	=&\left(\sum_{i=1}^{n}a_i(f^{\frac{1}{2}} \otimes 1_2) \otimes y_i^t\right)^*\left(\sum_{j=1}^{n}a_j(f^{\frac{1}{2}} \otimes 1_2) \otimes y_j^t\right) \ge 0 \qedhere 
	\end{split} 
	\end{equation*}
\end{proof}
\begin{prop}\label{tildeTDominatedByS}
	Let \(T: L_p(\mathcal{M}) \rightarrow L_1(M_k)\) be completely order bounded and let \(\tilde{T}\) be as in \eqref{defineTtilde}. Then there is a linear map \(S: L_p(\mathcal{M}) \otimes M_2 \rightarrow L_1(M_k) \otimes M_2\) such that \(S \pm \tilde{T}\) are completely positive and
	\begin{equation}\label{upperBoundForS}
	\begin{split}
		& 0 \le \left<S \left( \begin{bmatrix} 
		c & 0 \\ 0 & 0 \end{bmatrix} \right),
		\begin{bmatrix} 
		y & 0 \\ 0 & 0 \end{bmatrix} \right> \le 
		\|T\|_{cob} \|c\|_p \|y\|_{\infty} \\
		& 0 \le  \left<S \left( \begin{bmatrix} 
		0 & 0 \\ 0 & c \end{bmatrix} \right),
		\begin{bmatrix}
		0 & 0 \\ 0 & y \end{bmatrix} \right> \le 
		\|T\|_{cob} \|c\|_p \|y\|_{\infty}.
	\end{split}
	\end{equation}
	for all \(c \in L_p(\mathcal{M})_+\) and all \(y \in M_{k+}\).		
\end{prop}
\begin{proof}
	Let \((L_p(\mathcal{M}) \otimes M_2 \otimes M_{2k})_h\) denote the self adjoint part of \( L_p(\mathcal{M}) \otimes M_2 \otimes M_{2k}\) and let \(\varphi_T\) be the linear functional defined by \eqref{definePhiT}. For \(x \in (L_p(\mathcal{M}) \otimes M_2  \otimes M_{2k})_h\) we define
	\begin{equation}\begin{split}\label{defSublinearFunctional}
		\theta(x) = \inf \bigl\{\ &\|T\|_{cob} (\|f\|_p + \|g\|_p) + \varphi_T(a-b) \nobreakspace \big| \\
		&a, b \in (L_p(\mathcal{M}) \otimes M_2 \otimes M_{2k})_+, f, g \in L_p(\mathcal{M})_+, \\
		& x + a + b \le f \otimes \varepsilon_{11} \otimes 1_{2k}+ g \otimes \varepsilon_{22} \otimes 1_{2k} \bigr\}.
	\end{split}\end{equation}
	By Lemma \ref{fGreaterX}, the set on the right side of \eqref{defSublinearFunctional} is not empty, so \(\theta\) is well defined. We will show that \(\theta\) is sublinear. To do this, let \(x_1, x_2 \in  (L_p(\mathcal{M}) \otimes M_2 \otimes M_{2k})_h\) and \(\varepsilon > 0\). Then there exist \(a_1, b_1, a_2, b_2 \in (L_p(\mathcal{M}) \otimes M_2 \otimes M_{2k})_+\) and \(f_1, g_1, f_2, g_2 \in L_p(\mathcal{M})_+\) such that
	\begin{equation*}
	\begin{split}
		&x_1 + a_1 + b_1 \le f_1 \otimes \varepsilon_{11} \otimes 1_{2k} + g_1 \otimes \varepsilon_{22} \otimes 1_{2k}, \\
		&x_2 + a_2 + b_2 \le f_2 \otimes \varepsilon_{11} \otimes 1_{2k} + g_2 \otimes \varepsilon_{22} \otimes 1_{2k}, \\		&\theta(x_1) \ge \|T\|_{cob}(\|f_1\|_p + \|g_1\|_p) + \varphi_T(a_1 - b_1) - \varepsilon, \\
		&\theta(x_2) \ge \|T\|_{cob}(\|f_2\|_p + \|g_2\|_p) + \varphi_T(a_2 - b_2) - \varepsilon.
	\end{split}
	\end{equation*}
	Then we get
	\begin{equation*}
		x_1 + x_2 + a_1 + a_2 + b_1 + b_2 \le (f_1 + f_2) \otimes \varepsilon_{11} \otimes 1_{2k} + (g_1 + g_2) \otimes \varepsilon_{22} \otimes 1_{2k}. 
	\end{equation*}
	This implies
	\begin{equation*}
		\theta(x_1 + x_2) \le \theta(x_1) + \theta(x_2) - 2\varepsilon.
	\end{equation*}
	Since \(\varepsilon\) was arbitrary, \(\theta\) is sub-additive. Similarly, we show for \( 0 < \lambda \in \mathbb{R}\) and \(x \in  (L_p(\mathcal{M}) \otimes M_2 \otimes M_{2k})_h\) that
	\begin{equation*}
		\theta(\lambda x) \le \lambda \theta(x)
	\end{equation*}
	and hence
	\begin{equation*}
		\lambda \theta(x) = \lambda \theta( \frac{1}{\lambda} \lambda x) \le \lambda \frac{1}{\lambda} \theta(\lambda x) = \theta(\lambda x).
	\end{equation*}
	It remains to show that \( \theta(0) = 0\). If we put \( x = a = b = 0, f = g = 0\) in \eqref{defSublinearFunctional}, we get \(\theta(0) \le 0\). Lemma \ref{tCobDominatesPhiAMinusB} states that every element in the set on the right side of \eqref{defSublinearFunctional} is not negative for \(x = 0\). Hence \(\theta(0) = 0\).
	By the Hahn-Banach theorem there is a real-linear functional
	\begin{equation}\label{hahn-banach}
		\psi: (L_p(\mathcal{M}) \otimes M_2 \otimes M_{2k})_h \rightarrow \mathbb{R}
	\end{equation}
	such that \(\psi(x) \le \theta(x)\) for all \(x \in (L_p(\mathcal{M}) \otimes M_2 \otimes M_{2k})_h\). Now we can extend \(\psi\) to a complex linear functional on \(L_p(\mathcal{M}) \otimes M_2 \otimes M_{2k}\) by putting \(\psi(x) = \frac{1}{2}\psi(x + x^*) + \frac{1}{2} \textrm{i} \psi(\textrm{i} x^* - \textrm{i} x)\) for \(x \in L_p(\mathcal{M}) \otimes M_2 \otimes M_{2k} \).
	For \(c \in (L_p(\mathcal{M}) \otimes M_2)_+, \|c\|_{p,2} \le 1\), by Theorem \ref{infIsMinInNorm}, there exists \(f \in L_p(\mathcal{M})_+\) such that
	\begin{equation*}
		f \otimes 1_2 - c \ge 0 \textrm{ and } \|f\|_p \le 1. 
	\end{equation*}
	For \(y \in (M_{2k})_+, 0 \le y \le 1_{2k}\), we put \(x = c \otimes y^t, a = b = 0\), apply \eqref{defSublinearFunctional}, and get
	\begin{equation}\label{psiLessTcob}
		\psi(c \otimes y^t) \le \theta(c \otimes y^t) \le 2\|T\|_{cob}\|f\|_p \le 2 \|T\|_{cob}.
	\end{equation}
	Then we put \(x = -c \otimes y^t, a =  c \otimes y^t, b = 0, f = g = 0\), apply \eqref{defSublinearFunctional}, and get
	\begin{equation}\label{psiLessPhiT}
		\psi(-c \otimes y^t) \le \theta(-c \otimes y^t) \le \varphi_T(c \otimes y^t)
	\end{equation}
	For \(x = -c \otimes y^t, a = 0, b = c \otimes y^t, f = g = 0\), we apply \eqref{defSublinearFunctional} and \eqref{hahn-banach} and get
	\begin{equation}\label{psiLessMinusPhiT}
		\psi(-c \otimes y^t) \le \theta(-c \otimes y^t) \le - \varphi_T(c \otimes y^t)
	\end{equation}
	We combine \eqref{psiLessTcob}, \eqref{psiLessPhiT},  \eqref{psiLessMinusPhiT}, and get 
	\begin{equation*}
		0 \le \psi(c \otimes y^t) \le 2\|T\|_{cob}.
	\end{equation*}
	By Theorem \ref{xAsBoundedOp} and \cite{Bl}, Proposition II.3.1.2, we can write \(c \in L_p(\mathcal{M}) \otimes M_2, \|c\|_{p,2} \le 1\) as sum \(c = c_1 - c_2 + \textrm{i}(c_3 - c_4)\) where \(c_i \ge 0\) and \(\|c_i\|_{p,2} \le 1\) for \(i = 1, \cdots,4\). A similar decomposition holds for \(y \in M_{2k}\). Thus the bilinear map
	\begin{equation*}
		B: L_p(\mathcal{M}) \otimes M_2 \times M_{2k} \to \mathbb{C}, B(c,y) = \psi(c \otimes y^t)
	\end{equation*}
	is bounded. Hence there is a linear map \(S: L_p(\mathcal{M}) \otimes M_2 \rightarrow L_1(M_k) \otimes M_2\) such that \(<S(c),y> = \psi(c \otimes y^t)\).
	Next we show that \( S \pm \tilde{T}\) are completely positive. By Proposition \ref{operatorIsPositive}, it suffices to show that for 
	\(n \in \mathbb{N}\), \( x \in (L_p(\mathcal{M}) \otimes M_2 \otimes M_n)_+, y \in (M_k \otimes M_2 \otimes M_n)_+\), where \(x = [x_{ij}]\) with \(x_{ij} \in L_p(\mathcal{M}) \otimes M_2\) and \( y = [y_{ij}] \in M_k \otimes M_2\) the expression \(<S_n(x) \pm \tilde{T}_n(x),y>\) is positive. Now we have
	\begin{equation}\label{snPlusminusTn}
	\begin{split}
		<S_n(x) \pm \tilde{T}_n(x),y> &= \sum_{i,j = 1}^{n}<S(x_{ij}) \pm \tilde{T}(x_{ij}), y_{ji}> \\ 
		&= \psi(\sum_{i,j=1}^{n} x_{ij} \otimes y_{ji}^t) \pm \varphi_T(\sum_{i,j=1}^{n} x_{ij} \otimes y_{ji}^t).
	\end{split}
	\end{equation}
	By  Lemma \ref{xijTensoryjiPositive}, \(\sum_{i,j=1}^{n} x_{ij} \otimes y_{ji}^t\) is positive. We put \(x = - \sum_{i,j=1}^{n} x_{ij} \otimes y_{ji}^t,\) \(a = -x, b = 0, f = g = 0\) in \eqref{defSublinearFunctional} and get
	\begin{equation}\label{psiPlusPhi}
		- \psi(\sum_{i,j=1}^{n} x_{ij} \otimes y_{ji}^t) \le \varphi_T(\sum_{i,j=1}^{n} x_{ij} \otimes y_{ji}^t).
	\end{equation}
	Similarly, we put \(x = \sum_{i,j=1}^{n} x_{ij} \otimes y_{ji}^t,a = 0, b = -x, f = g = 0\) and get
	\begin{equation}\label{psiMinusPhi}
	- \psi(\sum_{i,j=1}^{n} x_{ij} \otimes y_{ji}^t) \le - \varphi_T(\sum_{i,j=1}^{n} x_{ij} \otimes y_{ji}^t).
	\end{equation}
	The combination of \eqref{snPlusminusTn}, \eqref{psiPlusPhi}, and \eqref{psiMinusPhi} gives
	\begin{equation*}
		<S_n(x) \pm \tilde{T}_n(x), y> \ge 0. 
	\end{equation*}
	Now let \(c \in L_p(\mathcal{M})_+\) and \(y \in M_{k+}\). We put \(x = c \otimes \varepsilon_{11} \otimes y^t \otimes \varepsilon_{11}, a = b = 0, f = \|y\|_{\infty}\cdot c, g = 0\). Then \eqref{defSublinearFunctional} gives
	\begin{equation*}
		0 \le \left< S \left( \begin{bmatrix} 
		c & 0 \\ 0 & 0 \end{bmatrix} \right),
		\begin{bmatrix} 
		y & 0 \\ 0 & 0 \end{bmatrix} \right> = \psi(x) \le \theta(x) \le \|T\|_{cob}\|c\|_p\|y\|_{\infty}.
	\end{equation*}
	The second part of \eqref{upperBoundForS} is shown similarly.
\end{proof}
\begin{prop}\label{TdominatedByS1AndS2}
	Let \( k \in \mathbb{N}\) and \(T: L_p(\mathcal{M}) \rightarrow L_1(M_k)\) be completely order bounded. Then there are linear maps \(S_1, S_2: L_p(\mathcal{M}) \rightarrow L_1(M_k)\) such that the map
	\begin{equation*}
	\begin{split}
		\Phi:~ &L_p(\mathcal{M}) \otimes M_2 \rightarrow L_1(M_k) \otimes M_2 \\
		& \left[ \begin{array}{cc} 
		x_{11} & x_{12} \\ x_{21} & x_{22} 
		\end{array} \right] \mapsto
		\left[ \begin{array}{cc} 
		S_1(x_{11}) & T(x_{12}) \\ T(x_{21}^*)^* & S_2(x_{22})
		\end{array} \right]
	\end{split}
	\end{equation*}
	is completely positive and \(\|S_1\|, \|S_2\| \le \|T\|_{cob}\).
\end{prop}
\begin{proof}
	Let \(\tilde{T}\) be as in \eqref{defineTtilde}. By Proposition \ref{tildeTDominatedByS}, there is a linear map \(S:~ L_p(\mathcal{M}) \otimes M_2 \to L_1(M_k) \otimes M_2\) such that \(S \pm \tilde{T}\) are completely positive and for \(c \in L_p(\mathcal{M})_+, y \in M_{k+}\)
	\begin{equation*}
	\begin{split}
	& 0 \le \left<S \left( \begin{bmatrix} 
	c & 0 \\ 0 & 0 \end{bmatrix} \right),
	\begin{bmatrix}{} 
	y & 0 \\ 0 & 0 \end{bmatrix} \right> \le 
	\|T\|_{cob} \|c\|_p \|y\|_{\infty} \\
	& 0 \le  \left<S \left( \begin{bmatrix} 
	0 & 0 \\ 0 & c \end{bmatrix} \right),
	\begin{bmatrix} 
	0 & 0 \\ 0 & y \end{bmatrix} \right> \le 
	\|T\|_{cob} \|c\|_p \|y\|_{\infty}.
	\end{split}
	\end{equation*}
	The next steps are quite similar to the proof of \cite{ArKr}, Proposition 3.18. Let \(\alpha\) be the scalar \(1\times 2\) matrix \([1~0]\) and \(\beta\) be the scalar \(1\times 2\) matrix \([0~ 1]\), and let \(\alpha^*, \beta^*\) be the adjoined matrices. We put
	\begin{equation*}
		S_1: L_p(\mathcal{M}) \to L_1(M_k), 
		a \mapsto \alpha S \left( \begin{bmatrix} 
		a & 0 \\ 0 & 0 
		\end{bmatrix} \right) \alpha^*
	\end{equation*}
	and
	\begin{equation*}
		S_2: L_p(\mathcal{M}) \to L_1(M_k), 
		a \mapsto \beta S \left( \begin{bmatrix} 
		0 & 0 \\ 0 & a 
		\end{bmatrix} \right) \beta^*
	\end{equation*}
	Since symmetric multiplication with a matrix and its adjoint is completely positive, the maps \(S_1\) and \(S_2\) and completely positive. For \(a \in L_p(\mathcal{M})_+\), we have
	\begin{equation*}
	\begin{split}
		\|S_1(a)\|_1 &= <S_1(a),1_k> = \left<\alpha S \left(  \begin{bmatrix} 
		a & 0 \\ 0 & 0 
		\end{bmatrix} \right) \alpha^*,1_k\right> \\
		&= \left<S \left( \begin{bmatrix} 
		a & 0 \\ 0 & 0 
		\end{bmatrix} \right),\begin{bmatrix} 
		1 & 0 \\ 0 & 0 
		\end{bmatrix} \right> \le \|T\|_{cob}\|a\|_p.
		\end{split}
	\end{equation*}
	Hence \(\|S_1\| \le \|T\|_{cob}\). Similarly, \(\|S_2\| \le \|T\|_{cob}\). Let \(\gamma_1\) be the scalar \(2 \times 4\) matrix \(\left[ \begin{smallmatrix} 
	1 & 0 & 0 & 0\\ 0 & 0 & 0 & 1
	\end{smallmatrix} \right].\) For \(a = [a_{ij}] \in L_p(\mathcal{M}) \otimes M_2\) we have
	\begin{equation*}
		\gamma_1^* \begin{bmatrix} 
		a_{11} & a_{12} \\ a_{21} & a_{22} 
		\end{bmatrix} \gamma_1 =
		\begin{bmatrix} 
		a_{11} & 0 & 0 & a_{12} \\ 
		0      & 0 & 0 & 0 \\ 
		0      & 0 & 0 & 0 \\ 
		a_{21} & 0 & 0 & a_{22} 
		\end{bmatrix}
	\end{equation*}
	and the map
	\begin{equation*}
		\Phi_1: L_p(\mathcal{M})  \otimes M_2 \to L_p(\mathcal{M}) \otimes M_4, a \mapsto \gamma_1^* a \gamma_1
	\end{equation*}
	is completely positive. Next, we show that the map
	\begin{equation*}
	\begin{split}
		\Phi_2:~& L_p(\mathcal{M}) \otimes M_4 \rightarrow L_1(M_k) \otimes M_4 \\
		& \begin{bmatrix} 
		x_{11} & x_{12}   \\ 
		x_{21}  & x_{22} 
		\end{bmatrix} \mapsto
		\begin{bmatrix} 
		S(x_{11}) & \tilde{T}(x_{12})   \\ 
		\tilde{T}(x_{21})  & S(x_{22}) 
		\end{bmatrix} \textrm{ where } x_{ij} \in L_p(\mathcal{M}) \otimes M_2
	\end{split}
	\end{equation*}
	is completely positive. So let \(n \in \mathbb{N}\) and \(x = [x_{ij}] \in (L_p(\mathcal{M})\otimes M_4 \otimes M_n )_+\) where \( x_{ij} \in L_p(\mathcal{M}) \otimes M_4 \) . Then \(S_{2n}(x) \pm \tilde{T}_{2n}(x) \ge 0\), and, by Lemma \ref{orderSelfAdjoint}
	\begin{equation*}
		\begin{bmatrix} 
		S_{2n}(x) & \tilde{T}_{2n}(x) \\ 
		\tilde{T}_{2n}(x)  & S_{2n}(x) 
		\end{bmatrix}  \ge 0.
	\end{equation*}
	We write \(x\) as 
	\begin{equation*}
		x = \begin{bmatrix} 
		x_{ij11} & x_{ij12}   \\ 
		x_{ij21}  & x_{ij22} 
		\end{bmatrix} 
	\end{equation*}
	where \(x_{ijlm} \in L_p(\mathcal{M}) \otimes M_2, i,j \in \{1, \cdots,n\},~ l,m = 1,2.\) Then
	\begin{equation*}
		\begin{bmatrix} 
		S_{2n}(x) & \tilde{T}_{2n}(x) \\ 
		\tilde{T}_{2n}(x)  & S_{2n}(x) 
		\end{bmatrix} = 
		\begin{bmatrix} 
		S_n(x_{ij11}) & S_n(x_{ij12}) & \tilde{T}_n(x_{ij11}) & \tilde{T}_n(x_{ij12})\\ 
		S_n(x_{ij21}) & S_n(x_{ij22}) & \tilde{T}_n(x_{ij21}) & \tilde{T}_n(x_{ij22})\\ 
		\tilde{T}_n(x_{ij11}) & \tilde{T}_n(x_{ij12}) & S_n(x_{ij11}) & S_n(x_{ij12})\\ 
		\tilde{T}_n(x_{ij21}) & \tilde{T}_n(x_{ij22}) & S_n(x_{ij21}) & S_n(x_{ij22})\\ 
		\end{bmatrix}.
	\end{equation*}
	We multiply this matrix from left by the scalar matrix \(\gamma_2 = \begin{bmatrix} 
	1 & 0 & 0 & 0   \\ 
	0 & 0 & 0 & 1 \end{bmatrix}\)
	and from right the by its adjoined matrix \(\gamma_2^*\), and get
	\begin{equation}\label{sttsPositive}
		 0 \le \begin{bmatrix} 
		 S_n([x_{ij11}]) & \tilde{T}_n([x_{ij12}])\\ 
		 \tilde{T}_n([x_{ij21}]) & S_n([x_{ij22}])\\ 
		 \end{bmatrix}.
	\end{equation}	
	We multiply the matrix in \eqref{sttsPositive} from right by the scalar \(2n \times 2n\) matrix \(\gamma_3\) and from left by \(\gamma_3^*\), where \(\gamma_3\) has 1 at position \((l,m)\) when \((l,m) = (2i-1,i)\) or \((l,m) = (2i,n+i)\) for \(i \in \lbrace 1,\cdots,n \rbrace\) and 0 else. Then the resulting matrix is \(\Phi_{2,n}(x)\). This shows that \(\Phi_2\) is completely positive.
	Next we define the linear map
	\begin{equation*}
		\Phi_3:~ L_1(M_k) \otimes M_4 \rightarrow L_1(M_k) \otimes M_2, b \mapsto \gamma_1 b \gamma_1^*
	\end{equation*}
	where \(\gamma_1\) is the \(2 \times 4\) matrix used in the beginning of the proof. Then \(\Phi_3\) is completely positive. Since \( \Phi = \Phi_3 \circ \Phi_2 \circ \Phi_1\), \(\Phi\) is completely positive. This finishes the proof.
\end{proof}
For linear maps \(S_1, S_2, T: L_p(\mathcal{M}) \rightarrow L_1(\mathcal{N})\) we define
\begin{equation}\label{definitionS1S2T}
\begin{split}
	\begin{bmatrix} 
	S_1 & T \\  T^*  & S_2 
	\end{bmatrix} :~ & L_p(\mathcal{M}) \otimes M_2 \rightarrow L_1(\mathcal{N}) \otimes M_2 \\
	& \begin{bmatrix}
	x_{11} & x_{12} \\  x_{21}  & x_{22} 
	\end{bmatrix}	\mapsto
	\begin{bmatrix} 
	S_1(x_{11}) & T(x_{12}) \\  T(x_{21}^*)^*  & S_2(x_{22}) 
	\end{bmatrix}
\end{split}
\end{equation} 
\begin{lem}
	Let \(\mathcal{N}\) be injective, \(T: L_p(\mathcal{M}) \rightarrow L_1(\mathcal{N})\) be completely order bounded, \(m, n_1,\cdots n_m \in \mathbb{N}\), \(x_l \in (L_p(\mathcal{M}) \otimes M_2 \otimes M_{n_l})_+, y_l \in (\mathcal{N} \otimes M_2 \otimes M_{n_l})_+, \textrm{ for } l= 1, \cdots,m\) and \(\varepsilon > 0\). Then there exist completely positive maps \(S_1, S_2: L_p(\mathcal{M}) \rightarrow L_1(\mathcal{N})\) such that \(\|S_1\|, \|S_2\| \le \|T\|_{cob}\) and for the map \(\left[ \begin{smallmatrix} S_1 & T \\  T^*  & S_2 \end{smallmatrix} \right]\) defined in \eqref{definitionS1S2T} holds
	\begin{equation*}
		\left< \begin{bmatrix} 
		S_1 & T \\  T^*  & S_2 \end{bmatrix}_{n_l}(x_l),y_l\right> \ge - \varepsilon, \quad l = 1,\cdots m.
	\end{equation*}
\end{lem}
\begin{proof}
	Let \(x_l = [x_{l,ij}]\), where \(x_{l,ij} \in L_p(\mathcal{M}) \otimes M_2\), and \(y_l = [y_{l,ij}]\), where \(y_{l,ij} \in \mathcal{N} \otimes M_2\), and for \(\ i, j = 1,\cdots n_l, \ l=1,\cdots,m\)
	\begin{equation*}
		x_{l,ij} = \begin{bmatrix}
		x_{l,ij11} & x_{l,ij12} \\  x_{l,ij21}  & x_{l,ij22} \end{bmatrix}, y_{l,ij} = \begin{bmatrix} 
		y_{l,ij11} & y_{l,ij12} \\  y_{l,ij21}  & y_{l,ij22} \end{bmatrix}. 
	\end{equation*}
	Since \(\mathcal{N}\) is injective, there exist by \cite{Bl}, Theorem IV.2.4.4, \(k \in \mathbb{N}\) and completely positive contractions \(\sigma_1: \mathcal{N} \rightarrow M_k\) and \(\sigma_2: M_k \rightarrow \mathcal{N}\) such that \(\sigma_1\) is continuous in the \(\sigma\)-weak topology and 
	\begin{equation}\label{approximateT}
	\begin{split}
		\Bigg|  \sum_{i,j=1}^{n_l}& ( <T(x_{l,ij21}^*)^*,y_{l,ji12} - \sigma_2 \circ \sigma_1 (y_{l,ji12})>   \\
		& + <T(x_{l,ij12}),y_{l,ji21} - \sigma_2 \circ \sigma_1 (y_{l,ji21}) > ) \Bigg| < \varepsilon, \ l = 1,\cdots,m.
	\end{split}
	\end{equation}

	Since \(x_l\) and \(y_l\) are positive, we have \(x_{l,ij21}^* = x_{l,ji12}\) and \(y_{l,ij12}^* = y_{l,ji21}\) for all \(l=1,\cdots,m\). Thus, \(<T(x_{l,ij21}^*)^*,y_{l,ji12}>\) and \(<T(x_{l,ji12}),y_{l,ij21}>\) are conjugate complex numbers for all \(i,j = 1, \cdots,n_l, \ l = 1,\cdots,m\), and the sum on the left side of this equation is a real number. Hence the sum is greater than \(-\varepsilon\). 
	
	Since \(\sigma_1\) is continuous in the \(\sigma\)-weak topology, its adjoint map \(\sigma_1^t\) maps \(L_1(M_k)\) to \(L_1(\mathcal{N})\) and is a completely positive contraction. Similarly, the adjoint map \(\sigma_2^t\) maps \(L_1(\mathcal{N})\) to \(L_1(M_k)\) and is a completely positive contraction, and \(\sigma_2^t \circ T\) is completely order bounded with \(\|\sigma_2^t \circ T\|_{cob} \le \|T\|_{cob}\). Now, we apply Proposition \ref{TdominatedByS1AndS2} and get linear maps \(S_1' , S_2': L_p(\mathcal{M}) \rightarrow M_k\) such that \(\|S_1'\|, \|S_2'\| \le \|T\|_{cob}\) and \(\left[ \begin{smallmatrix} S_1' & \sigma_2^t \circ T \\  \sigma_2^t \circ T^*  & S_2' \end{smallmatrix} \right]\) is completely positive. We put \(S_1 = \sigma_1^t \circ S_1'\) and \(S_2 = \sigma_1^t \circ S_2'\). Then \(\|S_1\|, \|S_2\| \le \|T\|_{cob}\), \(S_1\) and \(S_2\) are completely positive, and for all \(l=1,\cdots,m\)
	\begin{equation}\label{sumWithSigmas}
	\begin{split}
		\left<\begin{bmatrix}
		S_1 & T \\ T^* & S_2
		\end{bmatrix}_{n_l} (x_l), y_l \right> &= \sum_{i,j=1}^{n_l}  \big(<S_1(x_{l,ij11}),y_{l,ji11}> \\
		& \qquad\qquad + <T(x_{l,ij12}), y_{l,ji21}> \\
		& \qquad\qquad + <T(x_{l,ij21}^*)^*, y_{l,ji12}> \\
		& \qquad\qquad + <S_2(x_{l,ij22}), y_{l,ij22}> \big) \\
		&= \sum_{i,j=1}^{n_l} \big( <\sigma_1^t \circ S_1'(x_{l,ij11}),y_{l,ji11}> \\
		& \qquad \qquad + <T(x_{l,ij12}), \sigma_2 \circ \sigma_1(y_{l,ji21})>  \\
		& \qquad \qquad + <T(x_{l,ij21}^*)^*, \sigma_2 \circ \sigma_1 (y_{l,ji12})> \\
		& \qquad \qquad + <\sigma_1^t \circ S_2'(x_{l,ij22}), y_{l,ij22}> \big) \\
		& \quad + \sum_{i,j=1}^{n_l} \big( <T(x_{l,ij21}^*)^*,y_{l,ji12} - \sigma_2 \circ \sigma_1 (y_{l,ji12})> \\
		& \qquad \qquad + <T(x_{l,ij12}),y_{l,ji21} - \sigma_2 \circ \sigma_1 (y_{l,ji21}) > \big).
	\end{split}
	\end{equation}
	Now the first sum of the last expression in (\ref{sumWithSigmas}) is equal to
	\begin{equation*}
		\left<(\sigma_1^t)_{2n_l} \circ \begin{bmatrix}
		S_1' & \sigma_2^t \circ T \\ (\sigma_2^t \circ T)^* & S_2'
		\end{bmatrix}_{n_l} (x_l), y_l \right>,
	\end{equation*}
	which is greater than or equal to 0, because it is a composition of two completely positive maps applied to a positive element. The second sum is greater than \(-\varepsilon\) by \eqref{approximateT}. This finishes the proof.
\end{proof}
\begin{thm}
	Let \(1 \le p \le \infty, \mathcal{M}\) and \(\mathcal{N}\) be von Neumann algebras, \(\mathcal{N}\) injective, and \(T: L_p(\mathcal{M}) \rightarrow L_1(\mathcal{N})\) be a completely order bounded map. Then there exist linear maps \(S_1, S_2: L_p(\mathcal{M}) \rightarrow L_1(\mathcal{N}), \|S_1\|, \|S_2\| \le \|T\|_{cob}\) such that the map
	\begin{equation*}
	\begin{split}
		\Phi: &~L_p(\mathcal{M}) \otimes M_2 \rightarrow L_1(\mathcal{N}) \otimes M_2 \\
		&\begin{bmatrix}
			x_{11} & x_{12} \\ x_{21} & x_{22}
		\end{bmatrix} \mapsto
		\begin{bmatrix}
		S_1(x_{11}) & T(x_{12}) \\ T(x_{21}^*)^* & S_2(x_{22})
		\end{bmatrix}
	\end{split}
	\end{equation*}
	is completely positive.
\end{thm}
\begin{proof}
	Let \(\mathcal{N}^*\) be the dual space of \(\mathcal{N}\) and \(\mathcal{B}(L_p(\mathcal{M}), \mathcal{N}^*)\) be the vector space of all bounded linear maps from \(L_p(\mathcal{M})\) to 	\(\mathcal{N}^*\). By \cite{Ta}, Theorem IV.2.3, \(\mathcal{B}(L_p(\mathcal{M}), \mathcal{N}^*)\) is isomorph to the dual space of \(L_p(\mathcal{M}) \otimes_{\gamma} \mathcal{N}\), where \(\|\cdot\|_{\gamma}\) denotes the projective tensor norm. The weak\(^*\)-topology on \(\mathcal{B}(L_p(\mathcal{M}), \mathcal{N}^*)\) is given by the seminorms \(|<S(x),y>|\), where \(x \in L_p(\mathcal{M}), y \in \mathcal{N},\) and \(S \in \mathcal{B}(L_p(\mathcal{M}), \mathcal{N}^*)\). Let \(\mathcal{U} = \lbrace S \in \mathcal{B}(L_p(\mathcal{M}), \mathcal{N}^*) \  | \ \|S\| \le \|T\|_{cob} \rbrace\). By the Banach-Alaoglu theorem \(\mathcal{U}\) is compact in the weak\(^*\)-topology. Hence \(\mathcal{U} \times \mathcal{U}\) is compact in the product weak\(^*\)-topology. For \(\varepsilon > 0, m \in \mathbb{N}, n_1,\cdots,n_m \in \mathbb{N}, l = 1, \cdots,m, X = \lbrace (x_l,y_l) | \ x_l \in (L_p(\mathcal{M}) \otimes M_2 \otimes M_{n_l})_+, y_l \in (\mathcal{N} \otimes M_2 \otimes M_{n_l})_+   \rbrace\), let
	\begin{equation*}
	\begin{split}
		K(X,\varepsilon) =  \Bigg\{ &(S_1,S_2) \in \mathcal{U} \times \mathcal{U} \ | \ S_1, S_2 \textrm{ completely positive, }  \\
		&  \left<\begin{bmatrix}
		S_1 & T \\ T^* & S_2
		\end{bmatrix}_{n_l} (x_l) , y_l \right> \ge -\varepsilon \textrm{ for all } (x_l,y_l) \in X \Bigg\}.
	\end{split}
	\end{equation*}
	We will show that \(K(X,\varepsilon)\) is closed in the weak\(^*\)-topology and hence compact. So let \((S_1,S_2)\) be in the closure of \(K(X,\varepsilon)\). Let \(\delta > 0\) and \((x_l,y_l) \in X\), where 
	\begin{equation*}
	\begin{split}
		&x_l = [x_{l,ij}], \ x_{l,ij} = \begin{bmatrix}
		x_{l,ij11} & x_{l,ij12} \\ x_{l,ij21} & x_{l,ij22}	\end{bmatrix} \in L_p(\mathcal{M}) \otimes M_2, \\
		&y_l = [y_{l,ij}], \ y_{l,ij} = \begin{bmatrix}
		y_{l,ij11} & y_{l,ij12} \\ y_{l,ij21} & y_{l,ij22}	\end{bmatrix} \in \mathcal{N} \otimes M_2.
	\end{split}
	\end{equation*}
	Then we can find \((S_{1\delta},S_{2\delta}) \in K(X,\varepsilon)\) such that
	\begin{equation*}
	\begin{split}
		&\sum_{i,j=1}^{n_l}|<S_1(x_{l,ij11})-S_{1\delta}(x_{l,ij11}), y_{l,ji11}> | < \delta \\
		&\sum_{i,j=1}^{n_l}|<S_2(x_{l,ij22})-S_{2\delta}(x_{l,ij22}), y_{l,ji22}> | < \delta,
	\end{split}
	\end{equation*}
	and therefore
	\begin{equation*}
	\begin{split}
		&\left| \left<\begin{bmatrix}
		S_1 & T \\ T^* & S_2
		\end{bmatrix}_{n_l} (x_l) , y_l \right> - \left<\begin{bmatrix}
		S_{1\delta} & T \\ T^* & S_{2\delta}
		\end{bmatrix}_{n_l} (x_l) , y_l \right> \right| \\
		= & \left| \sum_{i,j=1}^{n_l}<S_1(x_{l,ij11}) -S_{1\delta}(x_{l,ij11}),  y_{l,ji11}> \right. \\
		& \left. + \sum_{i,j=1}^{n_l}<S_2(x_{l,ij22}) -S_{2\delta}(x_{l,ij22}), y_{l,ji22}>\right| < 2\delta
	\end{split}
	\end{equation*}
	Now \(\left<\left[\begin{smallmatrix}
	S_{1\delta} & T \\ T^* & S_{2\delta}
	\end{smallmatrix}\right]_{n_l} (x_l) , y_l \right>\) is a positive number. Hence the imaginary part
	\begin{equation*}
		\left| Im \left<\begin{bmatrix}
		S_1 & T \\ T^* & S_2
		\end{bmatrix}_{n_l} (x_l) , y_l \right> \right| < 2\delta
	\end{equation*}
	and the real part
	\begin{equation*}
		Re \left<\begin{bmatrix}
		S_1 & T \\ T^* & S_2
		\end{bmatrix}_{n_l} (x_l) , y_l \right> > \left<\begin{bmatrix}
		S_{1\delta} & T \\ T^* & S_{2\delta}
		\end{bmatrix}_{n_l} (x_l) , y_l \right> - 2\delta \ge \varepsilon - 2\delta.
	\end{equation*}
	Since \(\delta\) is arbitrary, \(\left<\left[\begin{smallmatrix}
		S_1 & T \\ T^* & S_2
	\end{smallmatrix}\right]_{n_l} (x_l) , y_l \right>\) is a real number which is greater than or equal to \(-\varepsilon\). Similarly, we show that \(\|S_1\|, \|S_2\| \le \|T\|_{cob}\) and \(S_1, S_2\) are completely positive. So \((S_1,S_2) \in K(X,\varepsilon)\).
	
	The sets \(K(X,\varepsilon), m \in \mathbb{N}, X= \{(x_1,y_1),\cdots(x_m,y_m)\},  \varepsilon > 0\) have the finite intersection property. If we have sets \(K(X_l,\varepsilon_l), \ l=1,\cdots,m\), we put \(X = \bigcup_{l=1}^m X_l\) and \(\varepsilon = \min \lbrace \varepsilon_1,\cdots,\varepsilon_m \rbrace\). Then \(K(x,\varepsilon) \subseteq \cap_{l=1}^m K(X_l,\varepsilon_l)\). Combining the finite intersection property and the compactness in the weak*-topology, we get 
	\begin{equation*}
		\bigcap_{X,\varepsilon} K(X,\varepsilon) \ne \emptyset.
	\end{equation*}
	Any pair \((S_1,S_2)\) in this intersection has the property that \(\left[\begin{smallmatrix}
	S_1 & T \\ T^* & S_2
	\end{smallmatrix}\right]\) is completely positive and \(\|S_1\|, \|S_2\| \le \|T\|_{cob}\). Since there is a projection \(z \in \mathcal{N}^{**}\) which works as a  projection from \(\mathcal{N}^*\) to \(L_1(\mathcal{N})\), the maps \(S_1' = zS_1\) and \(S_2 = zS_2\) have the desired properties.
\end{proof}

We close this section with an example of a completely order bounded map which is not decomposable when \(q > 2p\). For \(m \in \mathbb{N}\), let \(l_q^m\) be the vector space \(\mathbb{C}^m\) equipped with norm \(\|c\|_q = \left( \sum_{i=1}^{m}|c_i|^q\right)^{1/q}\) where \(c = (c_1,\cdots,c_m)\). Let
\begin{equation}\label{standardBase}
	e_1,\cdots,e_m \textrm{ be the standard base of } l_q^m.
\end{equation}
Before we show our example, we need some formulas to estimate the completely order bounded norm and the decomposable norm for linear maps from \(L_p(\mathcal{M})\) to \(l_q^m\).
\begin{lem}\label{upperBoundForTcob}
	Let \(1 \le p < \infty, 1 \le q < \infty, \frac{1}{p} + \frac{1}{p'} = 1, m \in \mathbb{N}, g_1,\cdots,g_m \in L_{p'}(\mathcal{M})\) and 
	\begin{equation*}
	T: L_p(\mathcal{M}) \rightarrow l_q^m, \quad f \mapsto \sum_{k=1}^{m}<f,g_k>e_k.
	\end{equation*}
	Then
	\begin{equation}\label{upperBoundTcob}
	 \|T\|_{cob} \le \sup \left\lbrace \left( \sum_{k=1}^{m} \|ag_kb\|_1^q \right) ^{1/q} \mid a,b \in L_{2p'}(\mathcal{M}), \|a\|_{2p'}, \|b\|_{2p'} \le 1 \right\rbrace.
	\end{equation}
\end{lem}
\begin{proof}
	Let \(n \in \mathbb{N} \) and \( x \in L_p(\mathcal{M}) \otimes M_n, \|x\|_{p,n} \le 1\). By Theorem \ref{infIsMinInNorm} and Theorem \ref{xAsBoundedOp}, there exist \(a, b \in L_{2p'}(\mathcal{M}), \|a\|_{2p'}, \|b\|_{2p'} \le 1, y \in \mathcal{M}\otimes M_n, \|y\|_{\infty}\le 1\) such that \(x = (b\otimes 1_n)y(a\otimes 1_n)\). Let \(\varepsilon_{ij}\) be as in \eqref{definitionEpsilonij} and \(y = \sum_{i,j=1}^{n} y_{ij} \otimes \varepsilon_{ij}.\)
	Then we have
	\begin{equation*}
		T_n(x) = \sum_{k=1}^{m}\sum_{i,j=1}^{n}<ag_kb,y_{ij}>e_k \otimes \varepsilon_{ij}.
	\end{equation*}
	For \(k = 1, \cdots,m\) let \(\varphi_k: \mathcal{M} \to \mathbb{C}, f \mapsto <ag_kb,f>.\)
	Then
 \(\varphi_k \)  is completely order bounded and \(\|\varphi_k\|_{cob} = \|\varphi_k\| = \|ag_kb\|_1 \). Since \(\|y\|_{\infty} \le 1\), we have
	\begin{equation*}
		0 \le \begin{bmatrix} \|\varphi_k\|1_n & (\varphi_k)_n(y) \\
		(\varphi_k)_n(y)^* & \|\varphi_k\|1_n \end{bmatrix}.
	\end{equation*}
	Then we get
	\begin{equation*}
	\begin{split}
		0 &\le \sum_{k=1}^{m} e_k \otimes \begin{bmatrix} \|\varphi_k\|1_n & (\varphi_k)_n(y) \\
		(\varphi_k)_n(y)^* & \|\varphi_k\|1_n \end{bmatrix} \\
		&= \begin{bmatrix} \sum_{k=1}^{m}\|\varphi_k\|e_k \otimes 1_n & T_n(x) \\ T_n(x)^* & \sum_{k=1}^{m}\|\varphi_k\|e_k \otimes 1_n \end{bmatrix}
	\end{split}
	\end{equation*}
	Hence
	\begin{equation*}
		\|T_n(x)\|_{q,n} \le \|\sum_{k=1}^{m}\|ag_kb\|_1e_k\|_q = (\sum_{k=1}^{m} \|ag_kb\|_1^q)^\frac{1}{q}
	\end{equation*}
	which proves \eqref{upperBoundTcob}.
\end{proof}
\begin{lem}\label{sumAlphaEpsiloniiLessAlpha}
	Let \(2 \le q < \infty\), \(\alpha \in M_n\), and  \(\varepsilon_{ii}\) be as in \eqref{definitionEpsilonij} for \(i=1,\cdots,n\). Then
	\begin{equation*}
	\sum_{i=1}^{n} \|\alpha\varepsilon_{ii}\|_q^q \le \|\alpha\|_q^q.
	\end{equation*}
\end{lem}
\begin{proof}
	Let \(\alpha^*\alpha = \beta = [\beta_{ij}]_{i,j=1}^n \). Let \( |\alpha\varepsilon_{ii}|\) be the positive matrix of the polar decomposition of \(\alpha\varepsilon_{ii} \). Then \(|\alpha\varepsilon_{ii}|^2 = \varepsilon_{ii} \beta \varepsilon_{ii} = \beta_{ii}\varepsilon_{ii} \) which implies \(|\alpha\varepsilon_{ii}|^q = \beta_{ii}^{q/2}\varepsilon_{ii} \) and
	\begin{equation*}
	\sum_{i=1}^{n} \|\alpha\varepsilon_{ii}\|_q^q = \sum_{i=1}^{n} Tr(\beta_{ii}^{q/2}\varepsilon_{ii}) = \sum_{i=1}^{n} \beta_{ii}^{q/2}. 
	\end{equation*}
	Since \(\beta\) is positive, it has eigenvalues \(\lambda_1, \cdots, \lambda_n \ge 0\), and there is a unitary matrix \(u=[u_{ij}]\in M_n\) such that \(\beta = u^*diag(\lambda_1,\cdots,\lambda_n)u \). Especially, we have \(\beta_{ii} = \sum_{l=1}^{n}|u_{li}|^2\lambda_l\). From \(u\) being unitary, it follows that \(\sum_{i=1}^{n}|u_{li}|^2 = 1 \) for \(l=1,\cdots,n\). Since \(q \ge 2\), the function \(\phi:[0,\infty]\rightarrow \mathbb{R}, t\mapsto t^{q/2}\) is convex. Hence we get
	\begin{equation*}
	\begin{split}
	\sum_{i=1}^{n} \beta_{ii}^{q/2} &= \sum_{i=1}^{n}\phi\Big(\sum_{l=1}^{n}\lambda_l|u_{li}|^2\Big)	\le
	\sum_{i=1}^{n} \sum_{l=1}^{n}|u_{li}|^2\phi(\lambda_l) \\
	&= \sum_{l=1}^{n} \lambda_l^{q/2} \sum_{i=1}^{n}|u_{li}|^2 =
	\sum_{l=1}^{n} \lambda_l^{q/2} = Tr(\beta^{q/2}) = \|\alpha\|_q^q. \qedhere
	\end{split}
	\end{equation*}
\end{proof}
For \(n \in \mathbb{N}\) let \(\varepsilon_{ij} \) be as in \eqref{definitionEpsilonij} and \(e_i\) as in \eqref{standardBase}. Then we define the linear map
\begin{equation}\label{definitionT}
	T: L_p(M_n) \rightarrow l_n^q, \quad T(x) = \sum_{i=1}^{n}Tr((\varepsilon_{1i}+\varepsilon_{i1})x)e_i.
\end{equation}
\begin{prop}\label{TcobEquals2}
	Let \(1 \le p < \infty, 2p < q < \infty\), and \(T\) as in \eqref{definitionT}. Then \(\|T\|_{cob} \le 2\).
\end{prop}
\begin{proof}
	Let \(\varepsilon > 0 \). By Lemma \ref{upperBoundForTcob}, there exist \(a,b \in L_{2p}(M_n), \|a\|_{2p}, \|b\|_{2p} \le 1\) such that
	\begin{equation*}
		\|T\|_{cob} \le \Bigg(\sum_{i=1}^{n}\|a(\varepsilon_{i1} + \varepsilon_{1i})b\|_1^q\Bigg)^{1/q} + \varepsilon.
	\end{equation*}
	Since for all \(i=1,\cdots,n\) we have \(\|a(\varepsilon_{i1} + \varepsilon_{1i})b\|_1 \le \|a\varepsilon_{i1}b\|_1 + \|a\varepsilon_{1i}b\|_1\) we get
	\begin{equation*}
		\Bigg(\sum_{i=1}^{n}\|a(\varepsilon_{i1} + \varepsilon_{1i})b\|_1^q\Bigg)^{1/q} \le
		\Bigg(\sum_{i=1}^{n}\|a\varepsilon_{i1}b\|_1^q\Bigg)^{1/q} +
		\Bigg(\sum_{i=1}^{n}\|a\varepsilon_{1i}b\|_1^q\Bigg)^{1/q}
	\end{equation*}
	Let \(\frac{1}{q}+\frac{1}{q'} = 1\). Then \(\|a\varepsilon_{i1}b\|_1 = \|a\varepsilon_{ii} \varepsilon_{i1}b\|_1 \le \|a\varepsilon_{ii}\|_q \cdot \|\varepsilon_{i1}b\|_{q'}\). Since \(q > 2p \ge 2\) we have \(q' < 2 \le 2p\). Hence there exist a real number \(s > 1\) such that \(\frac{1}{q'} = \frac{1}{2p} + \frac{1}{s}\). By the generalized H\"older's inequality, we have\(\|\varepsilon_{i1}b\|_{q'} \le \|\varepsilon_{i1}\|_s\|b\|_{2p} \le 1\). We apply Lemma \ref{sumAlphaEpsiloniiLessAlpha} and get
	\begin{equation*}
		\sum_{i=1}^{n}\|a\varepsilon_{i1}b\|_1^q \le \sum_{i=1}^{n}\|a\varepsilon_{ii}\|_1^q \le \|a\|_q^q.
	\end{equation*}
	Since \(q > 2p \ge 2\), we have \(\|a\|_q \le \|a\|_{2p} \le 1\). Hence	
	\begin{equation*}
		\Bigg(\sum_{i=1}^{n}\|a\varepsilon_{i1}b\|_1^q\Bigg)^{1/q} \le 1.
	\end{equation*}
	Similarly, we get
	\begin{equation*}
		\Bigg(\sum_{i=1}^{n}\|a\varepsilon_{1i}b\|_1^q\Bigg)^{1/q} \le 1.
	\end{equation*}
	Since \(\varepsilon\) was arbitrary, this finishes the proof.
\end{proof}
\begin{prop}\label{TdecgreaterHtnaSqrtN}
	Let \(1 \le p < \infty,~ 1 \le q < \infty\), and \(T\) as in \eqref{definitionT}. Then \(\|T\|_{dec} \ge n^{1/2q}\).
\end{prop}
\begin{proof}
	Since \(T\) is self-adjoint, we can apply \cite{ArKr}, Lemma 2.18 and 2.19, and get
	\begin{equation*}
		\|T\|_{dec} = \inf \{ \|S\| \mid S: L_p(M_n) \rightarrow l_q^n, S \pm T \textrm{ is completely positive} \}.
	\end{equation*}
	Let \(S:L_p(M_n) \rightarrow l_q^n\) be a linear map such that \(S \pm T\) is completely positive. There exist \(b_1,\cdots, b_n \in M_n \) such that \(S(f) = \sum_{k=1}^{n}<b_k,f>e_k\) for all \(f \in L_p(M_n)\). Since \(S \pm T\) are positive, we have 
	\begin{equation*}
		\sum_{k=1}^{n}<b_k \pm (\varepsilon_{1k}+ \varepsilon_{k1}),f>e_k \ge 0 \textrm{ for all } f\in L_p(M_n)_+.
	\end{equation*}
	This means that \(b_k \pm (\varepsilon_{1k}+ \varepsilon_{k1}) \ge 0\) for \(k=1,\cdots,n\). Let \(b_k = [b_{k,ij}]\). For \(k=1\) we have \(b_1-2\varepsilon_{11} \ge 0\) which implies \(b_{1,11} \ge 2\). For \(k > 1 \) we have
	\begin{equation}\label{matrixBk}
	\begin{split}
		0 &\le \begin{bmatrix}e_k^t & 0 \\ 0 & e_1^t\end{bmatrix}
		\begin{bmatrix}
		b_k & \varepsilon_{1k} + \varepsilon_{k1} \\
		\varepsilon_{1k} + \varepsilon_{k1} & b_k
		\end{bmatrix}
		\begin{bmatrix}e_k & 0 \\ 0 & e_1\end{bmatrix} \\
		&= \begin{bmatrix}
		e_k^tb_ke_k & e_k^t\varepsilon_{k1}e_1 \\
		e_1^t\varepsilon_{1k}e_k & e_1^tb_ke_1
		\end{bmatrix} =
		\begin{bmatrix}
			b_{k,kk} & 1 \\ 1 & b_{k,11}
		\end{bmatrix}
	\end{split}
	\end{equation}
	We compute the determinant of the last matrix in \eqref{matrixBk} and get \(b_{k,kk} b_{k,11} \ge 1\). Since the diagonal elements of \(b_k\) are positive, we get \(b_{k,kk} \ge \frac{1}{b_{k,11}}\) for all \(k \ge 2\). For \(k=1,\cdots,n\) we have
	\begin{equation*}
		\|S(\varepsilon_{kk})\|_q^q = \Big\|\sum_{j=1}^{n}<b_j,\varepsilon_{kk}>e_j\Big\|_q^q = \Big\|\sum_{j=1}^{n}b_{j,kk} e_j\Big\|_q^q = \sum_{j=1}^{n}(b_{j,kk})^q.
	\end{equation*}
	For \(k = 1\) we have \(\|S(\varepsilon_{11})\|_q^q = \sum_{j=1}^{n}(b_{j,11})^q\). For \(k > 1\) we have \(\|S(\varepsilon_{kk})\|_q^q \ge (b_{k,kk})^q \ge \frac{1}{(b_{k,11})^q}\). Hence
	\begin{equation*}
		\|S\|^q \ge \max \left\lbrace \sum_{j=1}^{n} (b_{j,11})^q, \frac{1}{(b_{k,11})^q}, k=2,\cdots,n\right\rbrace.
	\end{equation*}
	If there is some \(k, 2 \le k \le n\) such that \(b_{k,11} \le n^{-1/2q}\), then
	\begin{equation}\label{oneBSmaller}
		\|S\|^q \ge \frac{1}{(b_{k,11})^q} \ge \sqrt{n}.
	\end{equation}
	If \(b_{k,11} \ge n^{-1/2q}\) for all \(k \ge 2\), we have 
	\begin{equation}\label{allBlarger}
	\begin{split}
		\|S\|^q &\ge \sum_{k=1}^{n}(b_{k,11})^q \ge 2^q + (n-1)n^{-1/2} \\
		&= \frac{2^q\sqrt{n} + n -1}{\sqrt{n}} \\
		&\ge \sqrt{n}
	\end{split}
	\end{equation}
	Combining \eqref{oneBSmaller} and \eqref{allBlarger}, we get
	\begin{equation*}
		\|S\| \ge n^{\frac{1}{2q}}.
	\end{equation*}
	Since \(S\) was arbitrary with \(S\pm T\) completely positive, this finishes the proof.
\end{proof}
Now we can show our counterexample. Let \(\mathcal{M} = \oplus_{k=1}^{\infty}M_k\). On \(\mathcal{M}\) we have the semifinite, normal, faithful trace \(\tau_1(x) = \sum_{k=1}^{\infty}Tr_k(x_k)\) where \(x = \oplus_{k=1}^{\infty} x_k\) and \(Tr_k\) is the usual trace on \(M_k\). Since for any projection \(e \in \mathcal{M},\) we have \(\tau_1(e)\ge 1\) and therefore every \(\tau_1\)-measurable operator affiliated with \(\mathcal{M}\) is a bounded operator. Hence for \(1 \le p < \infty,\) we have \(L_p(\mathcal{M})= \linebreak \lbrace x = \oplus_{k=1}^\infty x_k \mid \sum_{k=1}^{\infty}\|x_k\|_p < \infty\rbrace\) with norm \(\|x\|_p = (\sum_{k=1}^{\infty}\|x_k\|_p)^{1/p}\). Let \(\mathcal{N} = \oplus_{k=1}^\infty l_{\infty}^k\). Then \(L_p(\mathcal{N}) = \lbrace f = \oplus_{k=1}^\infty f_k \mid \sum_{k=1}^{\infty} \|f_k\|^p < \infty\rbrace\) with norm \(\|f\|_q = (\sum_{k=1}^{\infty}\|x_k\|_q)^{1/q}\). For \(1 \le p,q < \infty,~ q > 2p \) let
\begin{equation*}
	T: L_p(\mathcal{M}) \rightarrow L_q(\mathcal{N}), \quad T(\oplus_{k=1}^{\infty}x_k) = \oplus_{k=1}^{\infty} T_k(x_k)
\end{equation*}
where \(T_k\) is defined as in \eqref{definitionT}. We claim that \(T\) is completely order bounded and \(\|T\|_{cob} \le 2\). To show this, let \(n \in \mathbb{N}\) and \(x \in L_p(\mathcal{M} \otimes M_n), \|x\|_{p,n} \le 1 \). According to Theorem \ref{infIsMinInNorm}, there exist \(f, g \in L_(\mathcal{M})_+\) such that
\begin{equation*}
	\|f\|_p = \|g\|_p = \|x\|_{p,n} \textrm{ and } \begin{bmatrix}f \otimes 1_n & x \\ x^* & g \otimes 1_n\end{bmatrix} \ge 0.
\end{equation*}
Then we have \(x = \oplus_{k=1}^{\infty}x_k, x_k \in L_p(M_k) \otimes M_n, f = \oplus_{k=1}^{\infty} f_k, f_k \in L_p(M_k)_+\) and \(g = \oplus_{k=1}^{\infty} g_k, g_k \in L_p(M_k)_+\) with
\begin{equation*}
	\begin{bmatrix}f_k \otimes 1_n & x_k \\ x_k^* & g_k \otimes 1_n\end{bmatrix} \ge 0 \textrm{ for all } k \in \mathbb{N}.
\end{equation*}

By Proposition \ref{TcobEquals2}, we have \(\|T_k\|_{cob}\le 2\) for all \(k \in \mathbb{N}\). By definition of completely order boundedness, for all \(k \in \mathbb{N}\), there exist \(h_{1,k}, h_{2,k} \in (l_k^q)_+,\) such that
\begin{equation*}
	\begin{bmatrix}h_{1,k} \otimes 1_n & T_{k,n}(x_k) \\ T_{k,n}(x_k)^* & h_{2,k} \otimes 1_n\end{bmatrix}
 \ge 0 \textrm{ and } \|h_{1,k}\|_q, \|h_{2,k}\|_q \le \frac{1}{2}(\|f_k\|_p + \|g_k\|_p).
\end{equation*}
We put \(h_1 = \oplus_{k=1}^{\infty}h_{1,k}\) and \(h_2 = \oplus_{k=1}^{\infty}h_{2,k}\). Then \(\|h_1\|_q, \|h_2\|_q \le 2\) and
\begin{equation*}
	\begin{bmatrix}h_1 \otimes 1_n & T_n(x) \\ T_n(x)^* & h_2 \otimes 1_n\end{bmatrix} \ge 0
\end{equation*}
which shows that \(\|T\|_{cob}\le 2\). 

Next, suppose that \(T\) is decomposable. Then there exists a completely positive map \(S: L_p(\mathcal{M}) \rightarrow L_q(\mathcal{N}),\) such that \(S \pm T\) are completely positive. For every \(j \in \mathbb{N}\), the embedding \(I_j: L_p(M_j) \rightarrow L_p(\mathcal{M}), x \mapsto (\cdots,0,x,0,\cdots)\) and the projection \(P_j:L_q(\mathcal{N}) \rightarrow l_q^j,~ \oplus_{k=1}^{\infty} y_k \mapsto y_j\) are completely positive and have norm less than or equal to 1. We have \(T_j = P_j \circ T \circ I_j \) and put \(S_j = P_j \circ S \circ I_j\). Then \(S_j \pm T_j\) are completely positive. We apply Proposition \ref{TdecgreaterHtnaSqrtN}  and get
\begin{equation*}
	\infty > \|S\| \ge \|S_j\| \ge \|T_j\|_{dec} \ge j^{1/2q} \textrm{ for all } j \in\mathbb{N}
\end{equation*}
which gives a contradiction. Thus \(T\) is not decomposable.

Erwin Neuhardt\\
Faculty of Computer Science\\
Hochschule Schmalkalden\\
D-98574 Schmalkalden\\
Germany\\
e-mail: e.neuhardt@hs-sm.de
\end{document}